\documentclass[twoside,11pt]{article}

\usepackage[preprint]{jmlr2e}
\usepackage{natbib}

\usepackage{bm}
\usepackage{dsfont}
\usepackage{upgreek}
\usepackage{booktabs}
\usepackage{fancyhdr}
\usepackage{float}
\usepackage{enumitem}
\usepackage{url}
\usepackage{hyperref}
\hypersetup{hidelinks}

\usepackage{bbding}
\usepackage{todonotes}
\usepackage{mathrsfs}
\usepackage{amsmath,amsfonts,amssymb}

\usepackage[ruled,vlined,linesnumbered,inoutnumbered,algo2e]{algorithm2e}

\usepackage{graphics,graphicx,epsf,epstopdf,subfigure}
\graphicspath{{./Figures/}}
\ifpdf
\DeclareGraphicsExtensions{.eps,.pdf,.png,.jpg}
\else
\DeclareGraphicsExtensions{.eps}
\fi

\newcommand{\R}{\mathbb{R}}
\newcommand{\N}{\mathbb{N}}

\newcommand{\Rnp}{\mathbb{R}^{n\times p}} 
 
\newcommand{\Rpp}{\mathbb{R}^{p\times p}} 
\newcommand{\Rnn}{\mathbb{R}^{n\times n}}  
\newcommand{\Rnm}{\mathbb{R}^{n\times m}}

\newcommand{\cL}{\mathcal{L}}

\newcommand{\orth}{\mathbf{orth}}
\newcommand{\rank}{\mathrm{rank}}

\newcommand{\stiefel}{{\cal S}_{n,p}}

\newcommand{\tr}{\mathrm{tr}}

\newcommand{\zz}{^{\top}}

\newcommand{\st}{\mathrm{s.\,t.}\,\,} 
\newcommand{\ff}{_{\mathrm{F}}}
\newcommand{\fs}{^2_{\mathrm{F}}}

\newcommand{\iid}{i=1,\dotsc,d}

\newcommand{\iin}{i=1,\dotsc,n}
\newcommand{\sumiid}{\sum\limits_{i=1}^d}

\newcommand{\sumjjd}{\sum\limits_{j=1}^d}

\newcommand{\dkh}[1]{\left(#1\right)}
\newcommand{\hkh}[1]{\left\{#1\right\}}

\newcommand{\jkh}[1]{\left\langle#1\right\rangle}
\newcommand{\norm}[1]{\left\|#1\right\|}
\newcommand{\abs}[1]{\left|#1\right|}

\newcommand{\dsp}[2]{{\mathbf{D_p}}\left(#1,#2\right)}
\newcommand{\dist}[2]{{\mathbf{d_p}}\left(#1,#2\right)}
\newcommand{\Pj}[1]{{\mathbf{P}}\!_{#1}}
\newcommand{\Pv}[1]{{\mathbf{P}}^{\perp}_{#1}}
\newcommand{\dists}[2]{{\mathbf{d}^2_{\mathbf{p}}}\left(#1,#2\right)}
\newcommand{\dik}{\mathbf{d}_i^{(k)}}
\newcommand{\ditk}{\left(\mathbf{d}_i^{(k)}\right)^2}
\newcommand{\di}{\mathbf{d}_i}
\newcommand{\dit}{\mathbf{d}_i^2}
\newcommand{\diz}{\mathbf{d}_i^{(0)}}
\newcommand{\ditz}{\left(\mathbf{d}_i^{(0)}\right)^2}

\newcommand{\ditkp}{\left(\mathbf{d}_i^{(k+1)}\right)^2}

\newcommand{\dips}{\left(\mathbf{d}_i^+\right)^2}
\newcommand{\djps}{\left(\mathbf{d}_j^+\right)^2}

\DeclareMathOperator*{\argmin}{arg\,min}

\newtheorem{assumption}[theorem]{Assumption}

\ShortHeadings{Seeking Consensus on Subspaces in Federated PCA}{Wang, Liu, and Zhang}
\firstpageno{1}

\title{Seeking Consensus on Subspaces in Federated Principal Component Analysis}

\author{\name Lei Wang \email wlkings@lsec.cc.ac.cn \\
	\addr State Key Laboratory of Scientific and Engineering Computing \\
	Academy of Mathematics and Systems Science \\
	Chinese  Academy of Sciences \\
	and University of Chinese Academy of Sciences \\
	Beijing, China
	\AND 
	\name Xin Liu\thanks{Corresponding author.} \email liuxin@lsec.cc.ac.cn \\
	\addr State Key Laboratory of Scientific and Engineering Computing \\
	Academy of Mathematics and Systems Science \\
	Chinese  Academy of Sciences \\
	and University of Chinese Academy of Sciences \\
	Beijing, China
	\AND 
	\name Yin Zhang \email yinzhang@cuhk.edu.cn \\
	\addr School of Data Science \\
	The Chinese University of Hong Kong \\
	Shenzhen, China
}

\ifpdf
\hypersetup{
  pdftitle={FAPS},
  pdfauthor={Lei Wang, Xin Liu, and Yin Zhang}
}
\fi

\begin{document}

\maketitle

% REQUIRED
\begin{abstract}
	In this paper, we develop an algorithm for federated principal component analysis (PCA) 
	with emphases on both communication efficiency and data privacy.  Generally speaking, 
	federated PCA algorithms based on direct adaptations of classic iterative methods, 
	such as simultaneous subspace iterations (SSI), are unable to preserve data privacy,
	while algorithms based on variable-splitting and consensus-seeking, 
	such as alternating direction methods of multipliers (ADMM), lack in communication-efficiency.
	In this work, we propose a novel consensus-seeking formulation
	by equalizing subspaces spanned by splitting variables instead of equalizing variables themselves,
	thus greatly relaxing feasibility restrictions and allowing much faster convergence.
	Then we develop an ADMM-like algorithm with several special features to make it practically efficient, 
	including a low-rank multiplier formula and techniques for treating subproblems. 
	We establish that the proposed algorithm can better protect data privacy than classic methods 
	adapted to the federated PCA setting. 
	We derive convergence results, including a worst-case complexity estimate, 
	for the proposed ADMM-like algorithm in the presence of the nonlinear equality constraints.
	Extensive empirical results are presented to show that the new algorithm, 
	while enhancing data privacy, requires far fewer rounds of communication than 
	existing peer algorithms for federated PCA.
\end{abstract}

% REQUIRED

\section{Introduction}

	Principal component analysis (PCA) 
	is a fundamental and ubiquitous technique 
	for data analysis and dimensionality reduction \citep{Moore1981}
	with a wide and still rapidly growing variety of applications, such as
	%least squares data fitting \citep{Golub1965}, 
	image compression~\citep{Andrews1976}, 
	dictionary learning~\citep{Aharon2006},
	facial recognition~\citep{Turk1991}, 
	latent semantic analysis~\citep{Deerwester1990}, 
	matrix completion~\citep{Candes2009}, and so on.
	
	Let $A \in \Rnm$ be an $n\times m$ data matrix, properly pre-processed
	with $n$ features and $m$ samples where, without loss of generality,
	$n \leq m$ is always assumed (usually $n \ll m$).  To reduce data dimensionality,
	PCA is to find an orthonormal basis of a $p$-dimensional subspace in $\R^n$
	such that the projected samples on this subspace have the largest variance.
	%Often times, such a PCA is sufficient since it contains the most relevant information 
	%about the underlying system or dataset represented by the matrix $A$.
	In truly large-scale applications, computing PCA is practically affordable only for $p \ll n$. 
	Mathematically, PCA can be formulated as the following optimization problem,
	\begin{equation}
		\label{eq:opt-trace}
		\min\limits_{Z \in \stiefel} \quad f (Z) := -\dfrac{1}{2} \tr \dkh{ Z\zz AA\zz Z },
	\end{equation}  
	where $\stiefel := \{Z \in \Rnp \mid Z\zz Z = I_p\}$ 
	denotes the Stiefel manifold \citep{Wang2021multipliers}.

\subsection{Federated Setting}
	\label{subsec:distributed}
	
	To develop scalable capacities for PCA calculations in today's big-data environments, 
	it is critical to study algorithms that can efficiently and securely process 
	distributed and massively large-scale data sets.
%	For a given large-scale data matrix $A \in \Rnm$, 
%	we call each column of $A$ a sample with length $n$, 
%	and the number of samples is $m$ which is assumed to be far bigger than $n$.
	In this paper, we consider the following federated setting \citep{Mcmahan2017communication}: 
	the data matrix $A$ is divided into $d$ blocks, each containing a group of samples; 
	namely, $A = [A_1\; A_2\; \dotsb \; A_d]$,
	where $A_i \in \R^{n\times m_i}$ so that $m_1 + \dotsb + m_d = m$.  
	These submatrices $A_i$, $\iid$, are stored locally in $d$ locations, 
	possibly having been collected and owned by different clients,
	and all the clients are connected, directly or indirectly, to a designated center 
	which could either be a special-purpose server or just one of the clients.
	In this federated setting, to solve \eqref{eq:opt-trace} it appears that
	products of the form
	\begin{equation*}
		AA\zz Z = A_1A_1\zz Z + A_2A_2\zz Z + \cdots + A_dA_d\zz Z
	\end{equation*}																			
	need to be aggregated at the center after all the local products, $A_iA_i\zz Z$, are 
	computed by the individual clients.  Indeed, this is the case when one adapts a classic 
	method to the federated setting.													

%	The main goal of this paper is to develop an algorithm potentially useful 
%	under the above federated setting with high communication costs and privacy concerns.
	
	In evaluating federated algorithms, a key measure of performance is 
	the total amount of communications required by algorithms.  
	In general, during iterations heavy computations are mostly done 
	at the local level within each client, and communications occur between iterations 
	for the center to aggregate newly updated local information from all the clients.
	In this work, we consider that the amount of communication at each iteration 
	remains essentially the same throughout the calculation.  
	In this setting, the total amount of communication 
	overhead will be proportional to the total number of iterations taken by an algorithm.
	That is, we consider the most prominent measure of communication efficiency to be 
	the number of iterations required by an algorithm to reach a moderately high accuracy.

	Besides enhancing communication-efficiency, preserving the privacy of local data is also 
	a critical task in the federated setting, since in many real-world applications local data consist of 
	sensitive information such as personal medical or financial records \citep{Lou2017,Zhang2018a}.
	In this paper, we consider the following privacy scenario that will be called
	{\em intrinsic privacy} for the sake of convenience.
	
	\begin{definition}[Intrinsic Privacy] \label{def:privacy}
		Each client does not allow its privately owned data matrix $A_i A_i\zz$, 
		$i \in \{1, 2, \dotsc, d\}$, 
		to be revealed to any others including the center.  In particular, the center 
		should be prevented from obtaining $A_i A_i\zz$ based on quantities 
		shared by client $i$.
	\end{definition}

	In this intrinsic privacy situation, it is not an option to implement a pre-agreed encryption 
	or a coordinated masking operation. For an algorithm to preserve intrinsic privacy, 
	publicly exchanged quantities must be safe in the sense that the center, or anyone else, 
	will be unable to compute local-data matrix $A_iA_i\zz$ from such quantities.  We will soon show
	next that direct adaptations of classic methods such as SSI are not intrinsically private.

\subsection{Overview of Related Works}

\label{subsec:overview}

	The subject of computing PCA has been thoroughly studied over several decades 
	and various iterative algorithms have been developed.
	%mostly based on computing eigenpairs of the symmetric matrix $AA\zz$ (or $A\zz A$).
	We briefly review a small subset of algorithms closely related to the present work.  
%	In particular, we focus on those designed for distributed computing and capable of 
%	keeping data security.
	 
	Classical PCA algorithms are mostly based on the simultaneous subspace iteration (SSI)
	\citep{Rutishauser1970,Stewart1976,Stewart1981},
	whose procedure can be readily extended to the federated setting as follows.
	\begin{equation} \label{eq:SSI}
		Z^{(k + 1)} \in \orth \dkh{\sumiid Y_i^{(k)}}
		\mbox{~with~}
		Y_i^{(k)} =  A_i A_i\zz  Z^{(k)},
	\end{equation}
	where $\orth (M)$ refers to the set of orthonormal bases for the range space of $M$.
	Under the federated setting, 
	each client computes $Y_i^{(k)}$ using the local data for $\iid$ 
	and sends  the result to the center server,
	which aggregates all the local products to generate the next iterate
	and then sends it to all the clients.
	For convenience, we say that SSI requires one round of communications per iteration
	for such information exchange between the center server and clients. 
%	Indeed there are various practical implementations of SSI under distributed settings, 
%	such as \citep{Elgamal2015,Fellus2015,Raja2015,Suleiman2016,Lazcano2019}.
	The main drawback of SSI lies in its slow convergence under unfavorable conditions,
	leading to intolerably high communication costs in the federated environments.
	To improve the communication efficiency, 
	\citet{Li2021communication} proposes an accelerated version of federated SSI
	by alternating between multiple local subspace iterations and one global aggregation,
	which is called LocalPower.

	We observe that the above direct adaption of SSI to the federated setting results
	in a vulnerability of losing intrinsic privacy, as defined in Definition~\ref{def:privacy}.
	This is because that the shared quantities by each client have a linear relationship 
	with its local data.   The crux is that, for given pairs of $Z^{(k)}$ and $Y_i^{(k)}$,
	the second equation in \eqref{eq:SSI} provides a set of linear equations 
	for the ``unknown" $A_i A_i\zz$.
	Should the center server aims to recover the private data $A_i A_i\zz$, 
	it would only need to collect a sufficient number of publicly shared matrices 
	$\{Z^{(k)}\}$ and $\{ Y_i^{(k)} \}$, and then to solve the resulting linear system of 
	equations for $A_i A_i\zz$ as given in \eqref{eq:SSI}.
	Under a mild condition, $A_iA_i\zz$ would be uniquely and exactly determined 
	by solving a linear system.
	To be precise, we formalize the above argument into the following proposition
	whose proof is evident.
	\begin{proposition} \label{priv-SSI}
		Let $\mathcal{Z}_k = \left[ Z^{(1)}\; Z^{(2)}\; \dotsb \; Z^{(k)} \right] 
		\in \R^{n \times kp}$.
		Suppose that $k$ is sufficiently large so that $\mathrm{rank}(\mathcal{Z}^k) = n$.  
		Then for $i = 1, 2, \dotsc, d$, there holds
		\begin{equation*}
			A_iA_i\zz = 
			\left[ Y_i^{(1)}\; Y_i^{(2)}\; \dotsb \; Y_i^{(k)} \right]\mathcal{Z}_k\zz
			\left(\mathcal{Z}_k\mathcal{Z}_k\zz\right)^{-1}.
		\end{equation*}
	\end{proposition}
	To put it simply, the ``federated SSI" algorithm cannot preserve \emph{intrinsic privacy}.
	That is, under mild conditions, the center can recover local data matrices exactly by solving 
	linear systems of equations bases on shared quantities. 

	In practice, an approximation of $A_iA_i\zz$ can be discovered after very few 
	iterations in practice, which is illustrated by the following numerical instances.
	We randomly generate the test matrix $A \in \mathbb{R}^{n \times m}$ 
	with $n = 1000$ and $m = 10000$,
	and the number of computed principal components is set to $p = 100$.
	In Figure \ref{subfig:Iter_SSI}, we record how the KKT violation 
	$\|(I_n - Z^{(k)} (Z^{(k)})\zz) AA\zz Z^{(k)}\|_{\mathrm{F}}$
	and the reconstruction error $\|\Phi^{(k)} - A_1 A_1^{\top}\|_{\mathrm{F}}$ 
	reduces as the number of iterations increases.
	Here, $\Phi^{(k)}$ represents the solution to the following optimization problem,
	\begin{equation*}
		\min_{\Phi \in \Rnn} \norm{\Phi}\fs
		\quad \st \quad  \Phi \mathcal{Z}_k = %\mathcal{Z}_k\zz = 
		\left[ Y_1^{(1)}\; Y_1^{(2)}\; \dotsb \; Y_1^{(k)} \right], %\mathcal{Z}_k\zz,
	\end{equation*}
	where $\mathcal{Z}_k$ is defined as in Proposition \ref{priv-SSI}.
	We can observe that the local data can be restored to certain accuracy, say $10^{-5}$, 
	much faster than solving the PCA problem.
	Next, we fix $m = 10000$ and $p = 100$ with $n$ ranging from $1000$ to $5000$.
	We record the number of iterations required by SSI to reach $10^{-5}$ 
	in KKT violation accuracy or reconstruction error in Figure \ref{subfig:n_SSI}.
	Again, the number of iterations required to recover the data 
	is much less than to solve the PCA problem.
	
	\begin{figure}[ht!]
		\centering
		\subfigure[Reconstruction error of SSI]{
			\label{subfig:Iter_SSI}
			\includegraphics[width=0.4\textwidth]{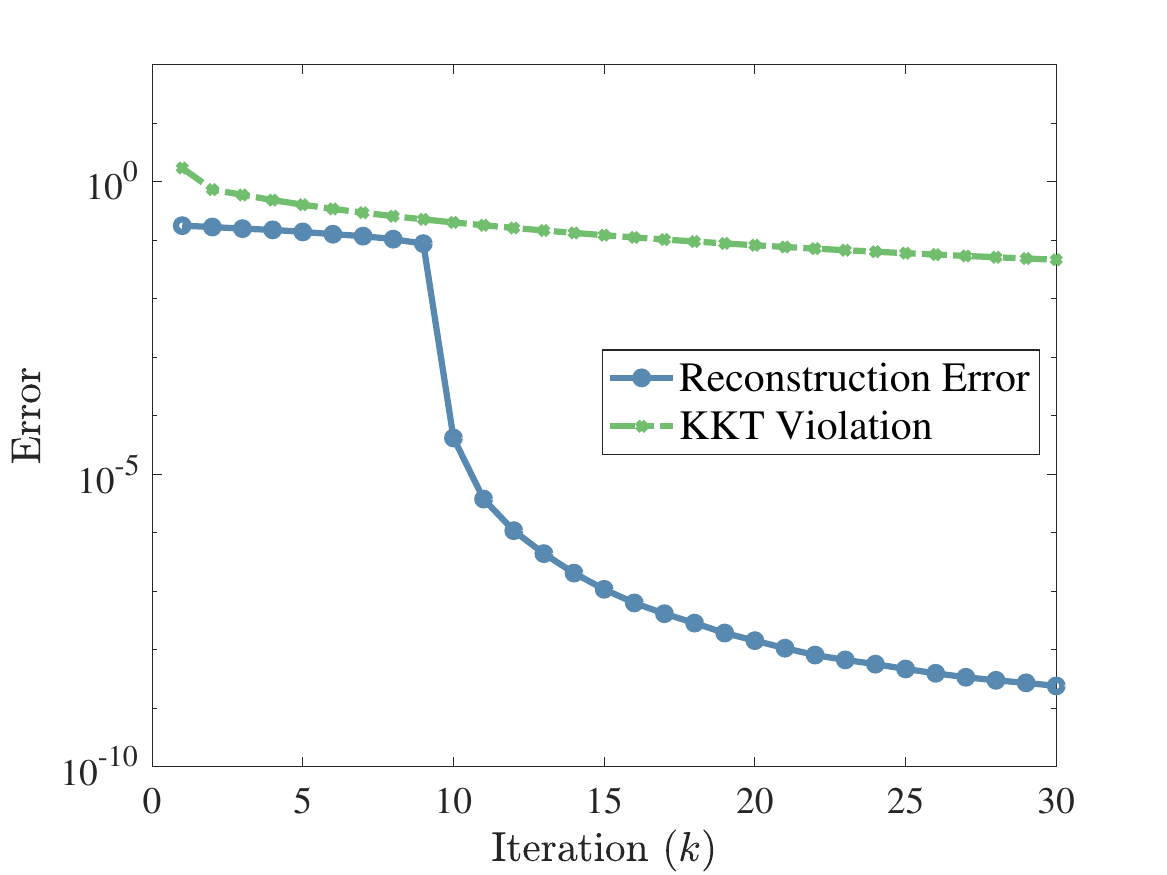}
		}
		\subfigure[Reconstruction iteration of SSI]{
			\label{subfig:n_SSI}
			\includegraphics[width=0.4\textwidth]{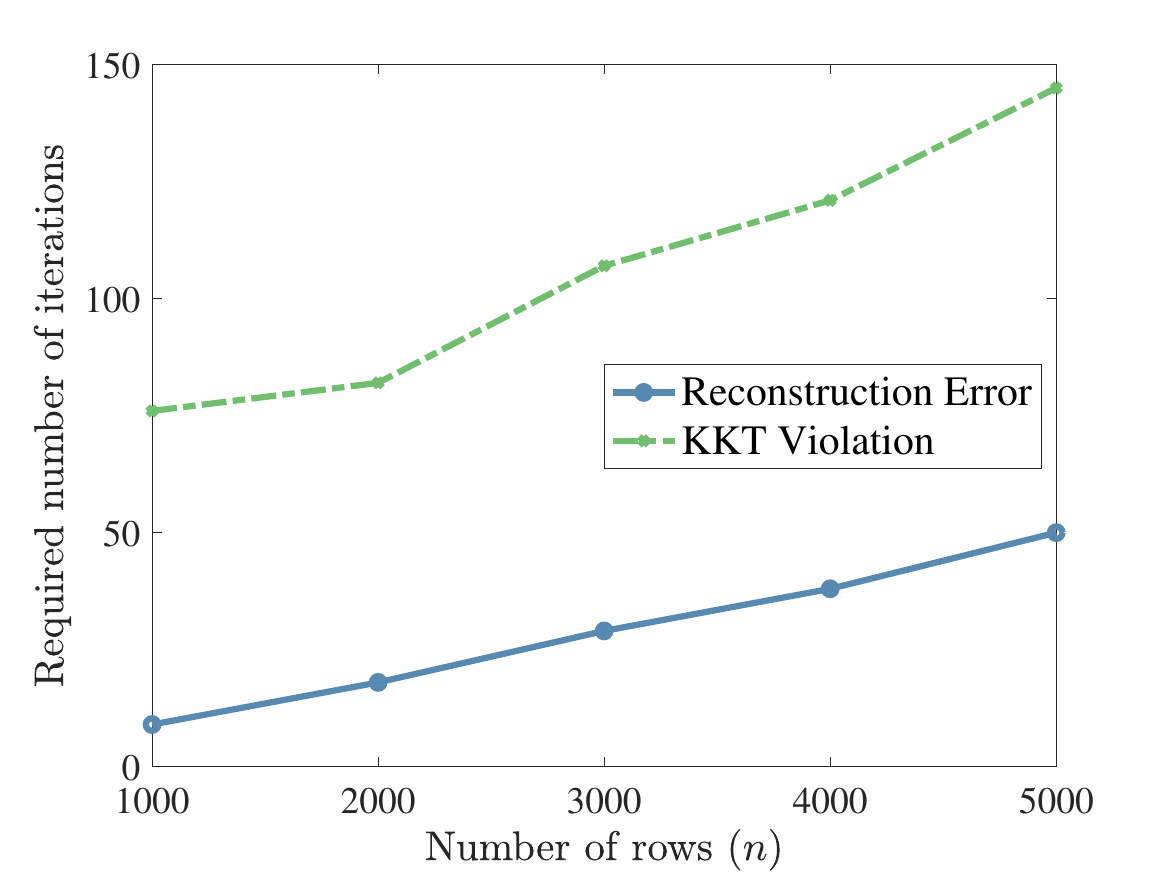}
		}
		
		\caption{Numerical examples of data leakage in the case of SSI.}
		\label{fig:security}
	\end{figure}

	To preserve the data privacy, \citet{Grammenos2020federated} 
	incorporate the differential privacy \citep{Dwork2014algorithmic} technique into the 
	federated PCA algorithm, which would cause a loss in accuracy loss 
	due to added noises.  \citet{Liu2019privacy} apply 
	the homomorphic encryption \citep{Acar2018survey} framework 
	to guarantee the data security in the context of federated PCA. 
	The homomorphic encryption framework involves large amounts of 
	computation and communication overheads due to the use of data inflation,
	resulting in significant performance degradation in practice.
	\citet{Chai2021federated} introduce a matrix masking scheme 
	designed for federated PCA to protect private data.
	However, it requires a trusted authority outside the federated system
	to generate secret masks and deliver them to the clients, which may 
	not exist or be desirable in general.  Moreover, this approach requires
	a high communication overhead since all the masked local data matrices
	have to be aggregated at the center server.
	
	Recently, a considerable amount of effort has been devoted to 
	developing decentralized algorithms for PCA.
	The decentralized extension of our algorithm is beyond the scope of this paper.
	We refer interested readers to 
	\citep{Schizas2015,Gang2019,Gang2021fast,Gang2022linearly,Ye2021,Wang2022decentralized}
	for more details.

\subsection{Our Contributions}
\label{subsec:contribution}

	We devise a communication-efficient approach, called FAPS, 
	to the federated PCA problem based on an ADMM-like framework.
	One iteration of FAPS is identical to a single iteration of the SSI algorithm
	applied to a matrix sum.  However, it differs from the federated SSI in that 
	the sum is not over the set of local data matrices $\{A_iA_i\zz\}$ as in \eqref{eq:SSI}, 
	but over a set of mask matrices $\{Q_i^{(k)}\}$.  That is,
	\begin{equation*} %\label{eq:FAPS}
		Z^{(k + 1)} \in \orth \dkh{\sumiid Y_i^{(k)}}
		\mbox{~with~}
		Y_i^{(k)} =  Q_i^{(k)} Z^{(k)},
	\end{equation*}
	where $Q_i^{(k)} \in \Rnn$ masks $A_i A_i\zz$ for $\iid$ at iteration $k$.
	Each mask matrix is computed locally by an individual client based on 
	its local data and other up-to-date information.  A main innovation of this 
	work is to construct these mask matrices from a novel projection-splitting 
	model along with an ADMM-like algorithm, which will be developed in Sections 2 and 3.
	
	The above mask operation brings two major advantages:
	(1) it empirically and significantly accelerates convergence rate in terms of iteration count,
	as is illustrated by the small numerical example in Table \ref{tb:toy} 
	(see Section~\ref{sec:numerical-result} for comprehensive numerical results); and 
	(2) it preserves intrinsic privacy (see Definition~\ref{def:privacy}) of local data, 
	in contrast to the intrinsic privacy vulnerability of the federated SSI approach 
	(see Propositions~\ref{priv-SSI} and \ref{priv-FAPS}).

	\begin{table}[ht!]
		\centering
		%\footnotesize
		\begin{tabular}{c|c|c} 
			\toprule 
			Algorithm & Iteration & Relative error\\
			\midrule
			SSI  &  207    &  1.09e-07 \\
			\midrule
			FAPS  &  42  &  8.04e-08 \\
			\bottomrule 
		\end{tabular} 
		\caption{Comparison of SSI and FAPS on a small example,
			where the matrix $A$, generated by \eqref{eq:gen-A} with $n = 2000$, $m = 128000$, 
			and $\xi = 1.01$, is tested with $p = 20$ and $d =128$.}
		\label{tb:toy} %\small
	\end{table}

	With the projection-splitting model, our ADMM-like framework is general and extendable.
	Beside SSI, many other existing methods for eigenspace calculation \citep[for example,][]{Liu2015b}
	can also be adapted to the federated PCA setting, but unfortunately with the same vulnerability in 
	terms of intrinsic privacy.   For each ``federated version" of such methods, our approach 
	can provide a corresponding ``masked version" to eliminate the intrinsic privacy vulnerability. 
	
	Furthermore, we established a theoretical convergence result 
	(see Theorem~\ref{thm:global}) for our FAPS algorithm which tackles
	not only non-convex manifold constraints
	but also nonlinear equality constraints that couple local variables to the global one.
	It is noteworthy that so far existing convergence theory for 
	ADMM algorithms \citep[for example,][]{Wang2019global,Zhang2020primal}
	on non-convex optimization over Riemannian manifolds are applicable
	only to coupled linear equality constraints.

\subsection{Notations}
\label{subsec:notation}

	We use $\R$ and $\N$ to denote the sets of real and natural numbers, respectively. The $p \times p$ identity matrix is represented by $I_p$.
	% \in \Rpp$.
	%The superscript $\top$ denotes the transpose operator for matrices.
	The Euclidean inner product of two matrices $Y_1$ and $Y_2$ of the same size is defined as $\jkh{Y_1, Y_2}=\tr(Y_1\zz Y_2)$, 
	where $\tr(B)$ is the trace of a square matrix $B$.
	The Frobenius norm and 2-norm of a matrix $X$ 
	are denoted by $\norm{X}\ff$ and $\norm{X}_2$, respectively.
	For a matrix $X$,
	the notation $\rank \dkh{X}$ stands for its rank;
	$\orth(X)$ refers to the set of orthonormal bases for its range space;
	%Given a differentiable function $g(X) : \Rnp \to \R$, the gradient of $g$ with respect to $X$ is denoted by $\nabla g(X)$. 
	and $\sigma_{\min} (X)$ denotes its smallest singular value. 
	%The set $\stiefel :=  \hkh{ X \in \Rnp \mid X\zz X = I_p}$ is referred to as the Stiefel manifold.
	For $X, Y \in \stiefel$, we define $\Pj{X} := XX\zz$, 
	$\Pv{X} := I_n - XX\zz$, $\dsp{X}{Y} := XX\zz - YY\zz$, 
	and $\dist{X}{Y} := \norm{ \dsp{X}{Y} }\ff$.
%	standing for, respectively, 
%	the projection operator onto the range space of $X$,
%	the projection operator onto the null space of $X\zz$,
%	the projection distance matrix and the projection distance.
	Other notations will be introduced at their first appearance. 

\subsection{Organization}

	The rest of this paper is organized as follows. 
	In Section \ref{sec:model}, we introduce a novel projection splitting model 
	with so-called subspace constraints, 
	and investigate the structure of associated Lagrangian multipliers.
	Then we propose a federated algorithm to solve this model 
	based on an ADMM-like framework in Section \ref{sec:algorithm}.
	Convergence properties of the proposed algorithm are studied in Section \ref{sec:convergence-analysis}.
	Numerical experiments on a variety of test problems are presented in Section \ref{sec:numerical-result} 
	to evaluate the performance of the proposed algorithm.
	We conclude the paper in the last section.
	%, we draw a brief conclusion and discuss some possible future developments.

\section{Projection Splitting Model}

\label{sec:model}

	We first motivate the proposed projection splitting model, then derive a low-rank 
	formula for Lagrangian multipliers associated with the so-called subspace constraints.
	This low-rank formula is essential to make the ADMM-like approach practical in its
	application to the proposed model with large-scale data matrices.

\subsection{Pursuit of an Optimal Subspace, Not Basis}

\label{subsec:ps}

	It is worth emphasizing that both the objective function $f$ and the feasible region $\stiefel$
	of problem \eqref{eq:opt-trace} are invariant under the transformation $Z \rightarrow ZO$ 
	for any orthogonal matrix $O\in\R^{p\times p}$.
	In essence, we are to pursue an optimal subspace rather than an optimal basis.
	Indeed, as is well-known, a global minimizer of \eqref{eq:opt-trace} can be any orthonormal basis matrix 
	for the optimal subspace spanned by the $p$ left singular vectors associated with the largest $p$ singular 
	values of $A$.  In addition, according to the discussions of \citet{Liu2013}, 
	the first-order stationarity condition of 
	\eqref{eq:opt-trace} can be expressed as follows.
	\begin{equation}
		\label{eq:sta-trace}
		\Pv{Z} AA\zz Z = 0 \mbox{~~and~~} Z \in \stiefel.
	\end{equation}

	As is mentioned earlier, we have a division of $A = [A_1\; A_2\; \dotsb\; A_d]$ 
	into $d$ column blocks and the $i$-th block $A_i$ is stored at client $i$.  
	Therefore, the objective function $f(Z)$ can be recast as a finite sum function.
	\begin{equation}\label{def:sum_fi}
		f(Z) = \sumiid f_i (Z) \mbox{~~with~~} f_i (Z) = -\dfrac{1}{2} \tr \dkh{ Z\zz A_iA_i\zz Z },
	\end{equation}
	where the $i$-th component of the objective function $f_i (Z)$ 
	can be evaluated only at client $i$ since $A_i$ is accessible only at client $i$. 
	% LX-revise-Do not forget
	%We assume $p < m_i$, for $\iid$, hereinafter.
	To derive a federated algorithm, we introduce a set of local variables, $\{X_i\}_{i = 1}^d$, 
	where, at client $i$, $X_i \in \stiefel$ is a local copy of the global variable $Z \in \stiefel$
	(here $X$ instead of $Z$ is used to avoid possible future confusion).

	At this point, the conventional approach would impose constraints to equalize, 
	one way or another, all the local variables $\{X_i\}_{i = 1}^d$ with the global variable $Z$.  
	For instance, one could formulate the following optimization problem with 
	a separable objective function.
	\begin{equation}\label{eq:opt-var}
		%\min\limits_{ \{X_i \in \stiefel\}_{i = 1}^d, \,  Z \in \stiefel } 
		\min\limits_{ X_i, \, Z \in \stiefel }  
		\quad \sumiid f_i(X_i) \quad \st \quad  X_i = Z, \;\; \iid.
	\end{equation}
	In this model, the set of variables is $( \{X_i\}_{i = 1}^d, Z )$ and $f_i$ is defined in \eqref{def:sum_fi}.
    	When an ADMM scheme is applied to this model,
	the subproblems corresponding to the local variables can all be solved simultaneously 
	and distributively.
%	This type of variable-splitting schemes, referred to as consensus problems in \citep{Boyd2011}, 
%	is essentially what was used by a recent distributed algorithm called D-PCA \citep{Schizas2015}, 
%	which will be used in our numerical comparison.
	
	However, we observe that the equalizing constraints in \eqref{eq:opt-var} require that all local 
	variables $\{X_i\}_{i = 1}^d$ must be equal to each other.  
	In other words, model \eqref{eq:opt-var} dictates that 
	every client must find exactly the same orthonormal basis for the optimal subspace,
	which is of course extremely demanding but totally unnecessary.  
	Under such severely restrictive constraints, a consensus is much harder to reach than 
	when each client is allowed to find its own orthonormal basis, independent of each other.
	
	To relax the restrictive equalizing constraints in \eqref{eq:opt-var}, 
	we propose a new splitting scheme that equalizes subspaces spanned by local variables 
	instead of the local variables (matrices) themselves.
	For this purpose, we replace the equalizing constraints in \eqref{eq:opt-var} by 
	$X_iX_i\zz = ZZ\zz$ for $\iid$. 
	Since both sides of the equations are orthogonal projections (recall $X_i, Z \in \stiefel$), 
	we call our new splitting scheme \textit{projection splitting}.
	The resulting projection splitting model is as follows.
	\begin{equation}\label{eq:opt-ps}
		%\min\limits_{ \{X_i \in \stiefel\}_{i = 1}^d, \,  Z \in \stiefel } 
		\min\limits_{ X_i, \, Z \in \stiefel }
		\quad  \sumiid f_i(X_i) \quad \st \quad  X_i X_i\zz = ZZ\zz, \;\; \iid.
	\end{equation}
	For ease of reference, we will call the constraints in \eqref{eq:opt-ps} \textit{subspace constraints}. 
	Obviously, these constraints are nonlinear and the optimization model \eqref{eq:opt-ps} is nonconvex.
	Conceptually, subspace constraints are easier to satisfy 
	than the variable splitting constraints $X_i = Z$.
	Computationally, however, subspace constraints do come with additional difficulties. 
	Since $X_iX_i\zz = ZZ\zz$ are large-size, 
	$n\times n$ matrix equations (compared to $n\times p$ in $X_i = Z$), 
	their corresponding Lagrangian multipliers are also large-size, $n\times n$ matrices.
	How to treat such large-size multiplier matrices is a critical algorithmic issue 
	that must be effectively addressed.

\subsection{Existence of Low-rank Multipliers}

\label{subsec:ex-low}

	By introducing dual variables, we derive a set of first-order 
	stationarity conditions
	for the projection splitting model \eqref{eq:opt-ps} in the following proposition, 
	whose proof will be given in Appendix \hyperref[apx:kkt-ps]{A}. %\ref{apx:kkt-ps}.
	
	\begin{proposition}
		\label{prop:kkt-multipliers}
		Let $( \{X_i \in \stiefel\}_{i = 1}^d, Z \in \stiefel )$ be a feasible point of the projection splitting model.
		Then $Z$ is a first-order stationary point of \eqref{eq:opt-trace} 
		if and only if there exist symmetric matrices $\Lambda_i \in \Rnn$,
		%with $\rank \dkh{\Lambda_i} \leq 2p$ 
		%\textcolor{red}{(non-necessary condition deleted)}
		$\Gamma_i \in \Rpp$, and $\Theta \in \Rpp$ so that %such that $( \{X_i\}, Z )$ satisfies 
		the following conditions hold:
		\begin{equation}\label{eq:kkt-multipliers}
				\sumiid \Lambda_i Z - Z \Theta = 0, \quad 
				 A_iA_i\zz X_i + X_i \Gamma_i + \Lambda_i X_i  = 0, \;\; \iid.
		\end{equation}

	\end{proposition}
	
	The equations in \eqref{eq:kkt-multipliers} along with the feasibility represent
	the KKT conditions for the projection splitting model \eqref{eq:opt-ps}.
	The dual variables $\Lambda_i \in \Rnn$, $\Gamma_i \in \Rpp$, and $\Theta \in \Rpp$ 
	are the Lagrangian multipliers associated with the equality constraints 
	$X_iX_i\zz = ZZ\zz$, $X_i\zz X_i = I_p$, and $Z\zz Z = I_p$, respectively.
	
	It is straightforward (but rather lengthy, see Appendix \hyperref[apx:kkt-ps]{A}) %\ref{apx:kkt-ps}.
	to verify that at any first-order stationary point $(\{X_i\}, Z)$ of \eqref{eq:opt-ps}
	(i.e., besides feasibility, \eqref{eq:sta-trace} also holds at $Z$), 
	the KKT conditions in \eqref{eq:kkt-multipliers} are satisfied by the following values of multipliers: 
	$\Theta = 0$, $\Gamma_i = -X_i\zz A_iA_i\zz X_i$, and
	\begin{equation}
		\label{eq:multiplier}
		\Lambda_i = - \Pj{X_i} A_iA_i\zz \Pv{X_i} 
		- \Pv{X_i}  A_iA_i\zz \Pj{X_i}, \quad \iid.
	\end{equation}
	Clearly, all $\Lambda_i$ satisfying \eqref{eq:multiplier} have a rank no greater than $2p$.  
	In fact, they are symmetrization of rank-$p$ matrices.
	As such, equation \eqref{eq:multiplier} provides a low-rank, closed-form formula 
	for calculating an estimated multiplier $\Lambda_i$ at a given $X_i$.  
	This formulation will play a prominent role in our algorithm, 
	for it effectively eliminates the costs of storing and updating $n\times n$ multiplier matrices.

	\begin{remark}
		\label{rmk:multipliers}
		We note that multipliers associated with the subspace constraints are non-unique.
		For example, in addition to \eqref{eq:multiplier}, the matrices
		\begin{equation*}
			\hat{\Lambda}_i = -A_iA_i\zz \Pv{X_i} - \Pv{X_i} A_iA_i\zz, \quad \iid,
		\end{equation*}
		also satisfy the KKT conditions in \eqref{eq:kkt-multipliers}.
		However, for $p \ll n$, the matrix $\Lambda_i$ in \eqref{eq:multiplier} has a much lower rank.
	\end{remark}

%%%%%%%%%%%%%%%%%
\section{Algorithm Development}
\label{sec:algorithm}

	In this section, we develop a federated algorithm to solve the projection splitting 
	model \eqref{eq:opt-ps} based on an ADMM-like framework.  Out of all the constraints,
	we only bring the subspace constraints in \eqref{eq:opt-ps}	
	into the augmented Lagrangian function:
	\begin{equation}\label{eq:lag-ps}
		\cL ( \{X_i\}, Z, \{\Lambda_i\} ) = \sumiid \cL_i (X_i, Z, \Lambda_i), 
	\end{equation}
	where for $\iid$,
	\begin{equation}\label{eq:Li}
		\cL_i (X_i, Z, \Lambda_i) =  f_i (X_i) 
		- \dfrac{1}{2} \jkh{ \Lambda_i, \dsp{X_i}{Z} } 
		+ \dfrac{\beta_i}{4} \dists{X_i}{Z},
	\end{equation}
	and $\beta_i > 0$ is a penalty parameter. The quadratic penalty term $\dists{X_i}{Z}$ 
	(see Section \ref{subsec:notation} for definition) measures the difference between 
	the two subspaces spanned by $X_i$ and $Z$, respectively.
	%\textcolor{red}{(2 references deleted here)}
	%by the projection distance in Frobenius norm (e.g., see \citep{Edelman1998,Yang2007}).

	Conceptually, at iteration $k$, our algorithm consists of the following three steps.
	\begin{enumerate}
		
		\item[(1)] Each client updates its own local variable to $X_i^{(k + 1)}$, $\iid$, that
		is an approximate solution to the local subproblem below, 
			\begin{equation*}
			X_i^{(k + 1)} \approx \argmin_{X_i \in \stiefel} \;
			\cL_i (X_i, Z^{(k)}, \Lambda_i^{(k)}).
			\end{equation*}
		
		\item[(2)] Each client implicitly updates its own multiplier to $\Lambda_i^{(k+1)}$ for $\iid$.
		
		\item[(3)] The center updates the global variable to $Z^{(k + 1)}$ 
		to make a progress towards solving the global subproblem:
			\begin{equation*}
			\min_{Z\in\stiefel} \;
			\cL ( \{X_i^{(k+1)}\}, Z,\{\Lambda_i^{(k+1)}\}).
			\end{equation*}
		
	\end{enumerate}
	The first two steps can be concurrently carried out in $d$ clients,
	while the last step requires communications between the center and all the clients.
	In the next three subsections, we specify in more concrete terms 
	how these three steps are carried out.   A detailed algorithm statement 
	will be given in Section \ref{subsec:framework}, and the issue of data security
	will be discussed in Section \ref{subsec:alg-sec}.

\subsection{Subproblems for Local Variables}

\label{subsec:sub-x}

	It is straightforward to derive that the subproblem for the local variables $\{X_i\}_{i = 1}^d$ has
	the following equivalent form.
	\begin{equation}
		\label{eq:ps-sub-x}
		\min\limits_{X_i \in \stiefel} \quad
		h_i^{(k)}(X_i) := - \dfrac{1}{2} \tr \dkh{ X_i\zz H_i^{(k)} X_i },
	\end{equation}
	where, for $\iid$, %$H_i^{(k)}$ is an $n \times n$ matrix defined by
	\begin{equation}
		\label{eq:Hik}
		H_i^{(k)} = A_iA_i\zz + \Lambda_i^{(k)} + \beta_i \Pj{Z^{(k)}}.
	\end{equation} 
	Clearly, \eqref{eq:ps-sub-x} is a standard eigenvalue problem where one computes
	a $p$-dimensional dominant eigenspace of an $n \times n$ real symmetric matrix.
	As a subproblem, \eqref{eq:ps-sub-x} needs not to be solved to a high precision.  
	In fact, we have discovered two inexact-solution conditions 
	that ensure both theoretical convergence and good practical performance.  
	It is important to note that using an iterative eigensolver, one does not need to compute 
	nor store the $n \times n$ matrix $H_i^{(k)}$ 
	since it is accessed through matrix-(multi)vector multiplications.

	The first condition is a sufficient reduction in function value.
	\begin{equation}
		\label{eq:ps-sub-x-con-1}
		h_i^{(k)} ( X_i^{(k)} ) - h_i^{(k)} ( X_i^{(k+1)} ) 
		\geq \dfrac{c_1}{c_1^{\prime}\norm{A_i}_2^2 + \beta_i}
		\norm{ \Pv{X_i^{(k)}} H_i^{(k)} X_i^{(k)} }\fs,
	\end{equation}
	where $c_1 > 0$ and $c_1^{\prime} > 0$ are two constants independent of $\beta_i$.
	This kind of conditions has been used to analyze convergence of iterative algorithms
	for solving trace minimization problems with orthogonality constraints \citep{Liu2013,Gao2018}.

	The second condition is a sufficient decrease in KKT violation.
	\begin{equation}\label{eq:ps-sub-x-con-2}
		\norm{ \Pv{X_i^{(k + 1)}} H_i^{(k)} X_i^{(k+1)} }\ff 
		\leq \delta_i \norm{ \Pv{X_i^{(k)}} H_i^{(k)} X_i^{(k)} }\ff,
	\end{equation}
	where $\delta_i \in [ 0, 1 )$ is a constant independent of $\beta_i$.
	This condition frequently appears in 
	inexact augmented Lagrangian based approaches \citep{Eckstein2013,Liu2019}. 
	It will play a crucial role in our theoretical analysis.
	
	The above two conditions, much weaker than optimality conditions of \eqref{eq:ps-sub-x}, 
 	are sufficient for us to derive global convergence of our ADMM-like framework.  
	In practice, it usually takes very few iterations of a certain iterative eigensolver, 
	such as SSI, LMSVD \citep{Liu2013} and SLRPGN \citep{Liu2015b},
	to meet these two conditions.

\subsection{Formula for Low-rank Multipliers}
\label{subsec:updating-multipliers}

	Now we consider updating the multipliers $\{ \Lambda_i \}_{i = 1}^d$
	associated with the subspace constraints in \eqref{eq:opt-ps}.
	In a regular ADMM algorithm, multiplier $\Lambda_i$ would be updated by a dual ascent step.
	\begin{equation*}\label{eq:dualascend}
		\Lambda_i^{(k+1)} = \Lambda_i^{(k)} - \tau_i\beta_i \dsp{X_i^{(k+1)}}{Z^{(k+1)}},
	\end{equation*}
	where $\tau_i > 0$ is a step size.  
	However, the above dual ascent step requires to store an $n \times n$ matrix at each client, 
	which can be prohibitive when $n$ is large.   
	%Using a low-rank approximation of the dual ascent step is an option,
	%but it requires considerable additional calculations.

	In our search for an effective multiplier-updating scheme, 
	we derived an explicit, low-rank formula~\eqref{eq:multiplier} in Section \ref{subsec:ex-low} 
	that is satisfied at any first-order stationary point, namely,
	\begin{equation}
		\label{eq:ps-mult}
		{\Lambda}_i^{(k+1)} = X_i^{(k+1)} (W_i^{(k+1)})\zz +  W_i^{(k+1)}(X_i^{(k+1)})\zz,
		%\text{ with } W_i^{(k+1)} = \dkh{ X_i^{(k+1)}(X_i^{(k+1)})\zz - I_n } A_iA_i\zz X_i^{(k+1)}.
	\end{equation}
	where, for $\iid$,
	\begin{equation}
		\label{eq:Wk}
		W_i^{(k+1)} = 	-\Pv{X_i^{(k+1)}} A_iA_i\zz X_i^{(k+1)}.
	\end{equation}
	With this low-rank expression, one can produce matrix-(multi)vector products 
	involving $\Lambda_i$ without any storage besides $X_i$ 
	(optionally one more $n \times p$ matrix $W_i$ for computational convenience).

	We note that $\Lambda_i$ in formula~\eqref{eq:ps-mult} is independent of the global variable $Z$.
	Thus, we choose to ``update" $\Lambda_i$ after $X_i$ and before $Z$.

\subsection{Subproblem for Global Variable}

\label{subsec:sub-z}

	The subproblem for the global variable $Z$ can also be 
	rearranged into a standard eigenvalue problem.
	\begin{equation}
		\label{eq:ps-sub-z}
		\min\limits_{Z\in\stiefel} \quad q^{(k)}(Z) := - \dfrac{1}{2} \tr \dkh{ Z\zz Q^{(k)} Z},
	\end{equation}
	where $Q^{(k)}$ is a sum of $d$ locally held matrices:
	\begin{equation}
		\label{eq:Qik}
		Q^{(k)} = \sumiid Q_i^{(k)} \mbox{ ~with~ }
		Q_i^{(k)} = \beta_i \Pj{X_i^{(k+1)}} - \Lambda_i^{(k+1)}.
	\end{equation}
	
	As is the case for local variables, we also approximately solve \eqref{eq:ps-sub-z}
	by an iterative eigensolver.  However, in the federated environment, 
	each iteration of a certain eigensolver  requires at least one round of communications.
	Therefore, in order to reduce the overall communication overheads,
	we employ a single iteration of SSI to inexactly solve the subproblem \eqref{eq:ps-sub-z}. 
	Starting from the current iterate $Z^{(k)} \in \stiefel$ stored at every client, 
	one computes
	\begin{equation}
		\label{eq:Zk+1}
		Z^{(k+1)} \in \orth\left(\sumiid Y_i^{(k)}\right) 
		\mbox{ with } Y_i^{(k)} = Q_i^{(k)} Z^{(k)},
	\end{equation} 
	which only invokes one round of communications per outer-iteration.
	Here, the local products $Y_i^{(k)} = Q_i^{(k)} Z^{(k)}$, $\iid$, 
	are calculated using the expressions for $Q_i^{(k)}$ and $\Lambda_i^{(k)}$,
	see \eqref{eq:ps-mult},\eqref{eq:Wk} and \eqref{eq:Qik},
%	\begin{equation}
%		\label{eq:QikZk}
%		\begin{aligned}
%			Q^{(k)}_i Z^{(k)} = {} & \beta_i  X_i^{(k+1)} (X_i^{(k+1)})\zz Z^{(k)} 
%			- X_i^{(k+1)} (W_i^{(k+1)})\zz Z^{(k)} \\
%			& - W_i^{(k+1)} (X_i^{(k+1)})\zz Z^{(k)},
%		\end{aligned}
%	\end{equation}
	which can be carried out distributively at each client with $O(np^2)$ 
	floating-point operations without actually forming any $n \times n$ matrices.
%	In a fully connected network,
%	the summation of these local products can be achieved by the all-reduce type of communication. 
%	More specifically, one can adopt the butterfly algorithm [36]. 
%	In this case, the communication overhead per iteration is $O\dkh{ np\log(d) }$.
	
	Comparing to the federated SSI in \eqref{eq:SSI}, we use $Q_i^{(k)}$ 
	in \eqref{eq:Zk+1} to mask the local data matrix $A_i A_i\zz$.  We will 
	see that this masking operation not only protects data privacy,
	but also significantly accelerates convergence, as will be
	empirically shown in Section \ref{sec:numerical-result}.

%	This strategy for solving subproblem~\cref{eq:ps-sub-z} of global variable $Z$ 
%	can significantly reduce computational costs and communication overheads,
%	while ensuring global convergence.

\subsection{Algorithm Description}

\label{subsec:framework}

	We now formally present the proposed algorithmic framework as Algorithm \ref{alg:FAPS} below, 
	named \textit{federated ADMM-like algorithm with projection splitting} 
	and abbreviated to FAPS.
	At iteration $k$, client $i$ first updates $X_i^{(k + 1)}$ and $W_i^{(k + 1)}$ using local data,
	and then computes $Y_i^{(k )}=Q_i^{(k )} Z^{(k)}$ that is transmitted to the center server.
	Finally, the center server updates $Z^{(k + 1)}$ and sends the results to all the clients. 
	This procedure is repeated until convergence.
	Upon termination, the final iterate will be an orthonormal basis 
	for an approximately optimal eigenspace of $AA\zz$.
	Same as the federated SSI \eqref{eq:SSI}, FAPS requires one round of communications per iteration.
%	from which the desired dominant SVD can be easily calculated by an extra procedure.

	\begin{algorithm2e}[ht!]
		%\SetAlgoLined
		%\color{blue}
		\label{alg:FAPS}
				
		\caption{Federated ADMM-like algorithm with projection splitting (FAPS).} 
		
		\KwIn{data matrix $A = [A_1, \dotsc, A_d]$, penalty parameters $\hkh{ \beta_i }$.}
		
		Set $k := 0$.
		Initialize $( \{X_i^{(0)}\}, Z^{(0)} )$ and 
		compute $\{\Lambda_i^{(0)}\}$ by \eqref{eq:ps-mult}.
		
		\While{``not converged"}
		{
			
			\For{each client $i \in \{1, 2, \dotsc, d\}$}
			{
				
				Find $X_i^{(k+1)} \in \stiefel$ that satisfies 
				\eqref{eq:ps-sub-x-con-1} and \eqref{eq:ps-sub-x-con-2}.
				
				Update the matrix $W_i^{(k+1)}$ by \eqref{eq:Wk}.
				
				Compute $Y_i^{(k)}$ in \eqref{eq:Zk+1} and send it to the center server. 
				
			}
		
			\For{the center server}
			{
			
				Update  $Z^{(k+1)} \in \stiefel$ by \eqref{eq:Zk+1} and sent it to all the clients.
				%Find $Z^{(k+1)} \in \stiefel$ that satisfies \eqref{eq:ps-sub-z-con}.
				
			}			
			Set $k := k + 1$.
		}
		\KwOut{$Z^{(k)}$.} %If requested, compute the dominant SVD of $A$ from $Z^{(k)}$.
	\end{algorithm2e}

\subsection{Preservation of Intrinsic Privacy of Local Data}
\label{subsec:alg-sec}

	In FAPS, at iteration $k$ the public information is the global variable value $Z^{(k)}$ and the 
	shared information from client $i$ is the product $Y_i^{(k)}= Q_i^{(k)}Z^{(k)}\in\R^{n \times p}$, 
	see \eqref{eq:Zk+1}.  Suppose that the center has access to all available information 
	$\{Z^{(k)}\}$ and $\{Y_i^{(k)}\}$ at all iterations.
	Then would it be possible for the center to recover any local data matrix $A_i A_i\zz$?
	
	First observe that the $n$ by $n$ mask matrix $Q_i^{(k)}$ varies from iteration to iteration
	and the available equation at iteration $k$, $Q_i^{(k)} Z^{(k)}=Y_i^{(k)}$, is $n$ by $p$ for $p\ll n$.  
	Hence, it is generally impossible to obtain a mask matrix from solving the associated under-determined 
	linear system of equations at the corresponding iteration.  By examining the expressions \eqref{eq:ps-mult},
	\eqref{eq:Wk} and \eqref{eq:Qik}, we can derive the relationship between a local data matrix
	and other involved quantities, whether known or unknown.
	
	\begin{proposition}\label{priv-FAPS}
	In Algorithm~1 at iteration $k$, for $\iid$ there holds
	\begin{equation}\label{eqn-QZ=Y}
		\left(\beta_i\Pj{X} + \Pv{X}A_iA_i\zz\Pj{X} + \Pj{X}A_iA_i\zz\Pv{X}\right)Z^{(k)} = Y_i^{(k)}
	\end{equation}
	for $X = X_i^{(k+1)}$, where only $Z^{(k)}, Y_i^{(k)} \in \R^{n\times p}$ are shared quantities. 
	\end{proposition}

	It is evident that equations in \eqref{eqn-QZ=Y} cannot be used to exactly solve for the local 
	data matrix $A_iA_i\zz$ without knowing sufficiently many local iterates $X_i^{(k+1)}$ 
	(beside $\beta_i$), which are all privately owned by client $i$.  That is, an exact recovery 
	of $A_iA_i\zz$ by the center from shared quantities is impossible in Algorithm~FAPS.
	
	In light of Propositions~\ref{priv-SSI} and \ref{priv-FAPS},  we conclude that, in contrast to
	the federated SSI algorithm, Algorithm~FAPS can preserve intrinsic data privacy,
	as defined in Definition~\ref{def:privacy}.

\section{Convergence Analysis}

\label{sec:convergence-analysis} 

	In this section, we rigorously establish 
	the global convergence of our proposed Algorithm \ref{alg:FAPS}
	under the following mild assumptions on the algorithm parameters.
	
	\begin{assumption}\label{asp:beta-svd}
		We assume the following conditions hold.
		
		(i) The algorithm parameter $\delta_i$ in \eqref{eq:ps-sub-x-con-2} satisfies
		%$0 \leq \delta_i < \underline{\sigma} / \sqrt{ 4 \rho d }$, $\iid$,
		\begin{equation*}
			0 \leq \delta_i < \dfrac{\underline{\sigma}}{2 \sqrt{\rho d}},  \;\; \iid,
		\end{equation*}
		where $\rho := \max_{i,j = 1, \dotsc, d}\left\{ \beta_i / \beta_j \right\} \geq 1$ 
		and $\underline{\sigma} := \sqrt{1 - 1 / (2\rho d)} \in (0, 1)$.\\
%		\begin{equation*}
%			\rho := \max_{i,j = 1, \dotsc, d}\left\{\frac{\beta_i}{\beta_j}\right\}
%			\mbox{ ~~and~~ }
%			\underline{\sigma} := \sqrt{1 - \frac{1}{2 \rho d}}.
%		\end{equation*}

		(ii) For a sufficiently large constant $\omega_i > 0$, the penalty parameter $\beta_i$ satisfies 
		\begin{equation*}
			\beta_i \geq \omega_i \norm{A}\fs,  \;\; \iid.
		\end{equation*}
	\end{assumption}
	\begin{remark}
		The above assumptions are imposed only for the purpose of theoretical analysis.
		An expression for $\omega_i$ will be given in Appendix \hyperref[apx:global]{B}. %\ref{apx:global}.
	\end{remark}
	
	We are now ready to present the global convergence and the worst-case complexity of FAPS.
	For brevity, we use the following simplified notations.
	\begin{equation}
		\label{eq:Dik}
		\mathbf{d}_i^{(k)} := \dist{X_i^{(k)}}{Z^{(k)}}, \mbox{~~for~~} \iid, \mbox{~~and~~} k \in \N.
	\end{equation}
	
	\begin{theorem}
		\label{thm:global}
		Let $X_i^{(0)} \in \stiefel$ and $Z^{(0)} \in \stiefel$ satisfy 
		\begin{equation}
			\label{eq:initial}
			\ditz \leq \frac{1}{\rho d},
			\quad \iid,
		\end{equation}
		and the sequence $\{ \{X_i^{(k)}\}_{i = 1}^d, Z^{(k)} \}$ be generated by Algorithm \ref{alg:FAPS}.
		%which starts from $( \{X_i^{(0)}\}_{i = 1}^d, Z^{(0)} )$.  
		Under Assumption \ref{asp:beta-svd},
		$\{ Z^{(k)} \}$ has at least one accumulation point,
		and any accumulation point is a first-order stationary point of problem~\eqref{eq:opt-trace}. 
		Moreover, there exists a constant $C>0$ 
		so that for any $N > 1$, it holds that 
		\begin{equation*}
			\min\limits_{k = 0, \dotsc, N-1} 
			%\max\limits_{\iid} 
			\hkh{
				\norm{\Pv{Z^{(k)}} AA\zz Z^{(k)} }\fs + \dfrac{1}{d} \sumiid \ditk 
			} 
			\leq \dfrac{C}{N}.
		\end{equation*}
	\end{theorem}
	The proof of this theorem, being quite long and tedious, is left to Appendix \hyperref[apx:global]{B}.
	%\ref{apx:global}.

\section{Numerical Experiments}
	
\label{sec:numerical-result}
	
	In this section, we evaluate the performance of FAPS 
	through comprehensive numerical experiments,
	which demonstrate its efficiency, robustness, and scalability.
	All the experiments are performed on a high-performance computing cluster 
	LSSC-IV\footnote{More information at \url{http://lsec.cc.ac.cn/chinese/lsec/LSSC-IVintroduction.pdf}.} 
	maintained at the State Key Laboratory of Scientific and Engineering Computing (LSEC), 
	Chinese Academy of Sciences.
	There are 408 nodes in the main part of LSSC-IV,
	and each node consists of two Intel Xeon Gold 6140 processors 
	(at $2.30$GHz $\times 18$) with $192$GB memory.
	The operating system of LSSC-IV is Red Hat Enterprise Linux Server 7.3.
	%All the algorithms in comparison are running in C++ with MPI for inter-process communication.	
	
	Note that FAPS is potentially useful under widely federated settings 
	with high communication costs and privacy concerns.
	The numerical experiments done on the cluster are just simulations 
	to observe the convergence rates of FAPS in comparison to others.
	{For our numerical results, 
	the number of iterations required by algorithms 
	(i.e., the speed of convergence) is the determining factor.}
	
%	We first describe our test problems in \cref{subsec:problem} 
%	and then give implementation details in \cref{subsec:impdetail}.
%	FAPS is compared with several recent methods that are briefly introduced 
%	in \cref{subsec:competitors}. 
%	We then report the experimental results in three main aspects:
%	(a) results on scalability in \cref{subsec:scalability};
%	(b) results on synthetic datasets in \cref{subsec:synthetic}; and
%	(c) results on popular image datasets in \cref{subsec:real}.
%	Finally, we conclude this section by discussing the issue of communication overhead 
%	in \cref{subsec:discussion}.
	
\subsection{Test Problems}
	
\label{subsec:problem}

	Two classes of test problems are used in our experiments.
	The first class consists of synthetic problems randomly generated as follows.
	We construct a test matrix $A\in\Rnm$ (assuming $n\leq m$ without loss of generality)
	by its (economy-form) singular value decomposition
	\begin{equation}
		\label{eq:gen-A}
		A = U \Sigma V\zz,
	\end{equation}
	where both $U\in\Rnn$ and $V\in\R^{m\times n}$ are orthonormalization
	of matrices whose entries are random numbers drawn independently,
	identically and uniformly from $[-1,1]$,
	and $\Sigma \in \Rnn$ is a diagonal matrix with diagonal entries 
	%are $\Sigma_{ii}=\xi^{1 - i}$, $\iin$,
	\begin{equation}
		\label{eq:Sigma_ii}
		\Sigma_{ii}=\xi^{1-i}, \quad \iin,
	\end{equation}  
	for a parameter $\xi > 1$ that determines the decay rate of the singular values of $A$.
	In general, smaller decay rates (with $\xi$ closer to 1) correspond to more difficult cases. %test cases.
	
	The second class of test problems consists of matrices from four popular image data sets 
	frequently used in machine learning, including 
	MNIST\footnote{Available from \url{http://yann.lecun.com/exdb/mnist/}.}, 
	Fashion-MNIST\footnote{Available from \url{https://github.com/zalandoresearch/fashion-mnist}.}, 
	CIFAR-10\footnote{Available from \url{https://www.cs.toronto.edu/~kriz/cifar.html}.\label{note:CIFAR}},
	and CIFAR-100$^{\ref{note:CIFAR}}$.
	%\footnote{Available from \url{https://www.cs.toronto.edu/~kriz/cifar.html} }.
	In both MNIST and Fashion-MNIST, the sample dimension is $n = 28 \times 28 = 784$
	and the number of sample is $m = 60000$;
	while in both CIFAR-10 and CIFAR-100, the sample dimension is $n=3 \times 32 \times 32 = 3072$
	with $m = 50000$.
	
\subsection{Implementation Details}
	
\label{subsec:impdetail}
	
	We use an adaptive strategy to tune the penalty parameters $\beta_i$,
	in which we periodically increase the penalty parameters value when 
	the projection distance has not seen a sufficient reduction. 
	Given initial values $\beta_i^{(0)}$, at iteration $k > 0$
	we first compute the projection distance $\dik$
	%between local variable $X_i^{(k)}$ 
	%and global variable $Z^{(k)}$ as follows:
	%\begin{equation*}
	%	\mathbf{d}_i^{(k)}
	%	= \norm{ D_i^{(k)} }\ff
	%	= \dkh{ 2p - 2 \norm{(X_i^{(k)})\zz Z^{(k)}}\fs }^{\frac{1}{2}}
	%	%\mbox{~~with~~} D_i^{(k)} = X_i^{(k)}(X_i^{(k)})\zz - Z^{(k)}(Z^{(k)})\zz,
	%\end{equation*}
	%where $D_i^{(k)}$ is defined in \cref{eq:Dik},
	and then update the penalty parameter by the recursion rule:
	\begin{equation}
		\label{eq:adap}
		\beta_i^{(k+1)} = 
		\begin{cases}
			\dkh{1+ \theta} \beta_i^{(k)}, & \mbox{if} \bmod(k,5)=0 %k\geq 6, 
			\mbox{~and~} \mathbf{d}_i^{(k - 5)} \leq \dkh{1 + \mu} \mathbf{d}_i^{(k)}, \\
			\beta_i^{(k)}, & \mbox{otherwise}.
		\end{cases}
	\end{equation}
	By default, we set $\beta_i^{(0)} = 0.15 \norm{A_i}_2^2$, 
	$\theta = 0.1$, and $\mu = 0.01$ in our implementation.
	
	We initialize the global variable $Z^{(0)}$ as orthonormalization of a random 
	$n \times p$ matrix whose entries follow the i.i.d.\!~uniform distribution in $[-1,1]$.  
	Then we set $X_i^{(0)} = Z^{(0)}$ ($\iid$).
	
	For solving subproblem~\eqref{eq:ps-sub-x} approximately,
	we choose to use SSI which, at outer iteration $k$, 
	generates an inner-iteration sequence
	$X_{i}^{(k)}(j)$ for $j = 0, 1, \dotsc$, with the warm-start $X_{i}^{(k)}(0) = X_i^{(k)}$.
	To reduce the computation costs, we use the following termination rule.
	\begin{equation}\label{eq:stop-x}
		\norm{ X_i^{(k)}(j) - X_i^{(k)}(j-1) }\ff
		\leq \epsilon_x \norm{ X_i^{(k)}(j) }\ff,
	\end{equation}	
	for a prescribed tolerance $\epsilon_x > 0$, which measures the relative change 
	between two consecutive inner iterates.   
	%Once \cref{eq:stop-x} is met at inner iteration $j$, we set 
	In our experiments, we set $\epsilon_x = 10^{-2}$ as the default value.
	For solving subproblem~\eqref{eq:ps-sub-z} approximately, starting from $Z^{(k)}$ 
	we take a single iteration of SSI to obtain $Z^{(k+1)}$, see \eqref{eq:Zk+1}.
	
	We terminate FAPS if either the following condition holds,
	\begin{equation}
		\abs{ \sumiid\norm{A_i\zz Z^{(k)}}\fs - \sumiid\norm{ A_i\zz Z^{(k - 1)} }\fs } 
		\leq 10^{-10} \sumiid\norm{A_i\zz Z^{(k)}}\fs,
		\label{eq:rela-kkt}
	\end{equation}
	or the maximum iteration number $\mathtt{MaxIter} = 3000$ is reached. 
	The condition \eqref{eq:rela-kkt} measures the relative change in objective function values.
%	We present a practical version of FAPS as \cref{alg:FAPS-practical} below,
%	where superscripts are omitted for brevity.
	
	\begin{remark}
		We choose not to use the KKT violation as the stopping criterion since it requires
		extra communication overheads under a federated environment.
	\end{remark}
	
	In our experiments, we collect and compare four performance measurements:
	wall-clock time, total number of iterations, 
	scaled KKT violation defined by
	\begin{equation*}
		\dfrac{1}{\norm{A}\fs} \norm{ \Pv{Z^{(k)}} AA\zz Z^{(k)} }\ff,
	\end{equation*}
	and relative error in singular values defined by 
	${\norm{ \Sigma^{(k)} - \Sigma^{\ast} }\ff} / \norm{\Sigma^{\ast}}\ff$,
	where the diagonal matrices $\Sigma^{\ast} \in \Rpp$ and $\Sigma^{(k)} \in \Rpp$ 
	hold, respectively, the exact and computed dominant singular values.
	
\subsection{Competing Algorithms}

\label{subsec:competitors}
	 
	We compare the performances of FAPS mainly with 
	two closely related but representative algorithms.
	The first competing algorithm is an adaptation of the classic 
	SSI algorithm \citep{Rutishauser1970,Stewart1976,Stewart1981} 
	to the federated setting described in Section~\ref{subsec:distributed}.
	The second competing algorithm is called LocalPower \citep{Li2021communication},
	which is an accelerated version of SSI developed in federated mode.
	Originally, LocalPower was designed to communicate after every $q$ local subspace iterations.
	In order to guarantee convergence, later LocalPower applies a decay strategy 
	to gradually decrease the number of local subspace iterations.
	Specifically, LocalPower halves $q$ every round of communications until it reaches $1$.
	As suggested by \citet{Li2021communication}, 
	we choose the initial value of $q$ as $8$.
	It is worth mentioning that LocalPower can not preserve the privacy of local data either,
	since it boils down to vanilla SSI after $q$ reaches $1$.
	In our experiments, we adopt the same initialization and stopping criterion 
	as described in Section \ref{subsec:impdetail}.
	
	We implement FAPS, SSI, and LocalPower in C++ 
	with MPI for inter-process communication
	to the best of our ability.   
	Unless otherwise specified, the communication is realized 
	by the all-reduce operations in MPI.
	In our implementation, we use the C++ linear algebra library {\it Eigen}\footnote
	{Available from \url{http://eigen.tuxfamily.org/index.php?title=Main_Page}} 
	(version 3.3.8) for matrix computations.  In particular, 
	orthonormalization of an $n \times p$ matrix is done via the (economy-size) 
	QR factorization at a cost of $O(np^2)$ operations.
	
\subsection{Comprehensive Comparison on Synthetic Data}
	
\label{subsec:synthetic}
	
	We now compare the performances of the three algorithms 
	on a variety of synthetic test problems,
	run under the aforementioned federated environment 
	with the number of computing clients fixed at $d = 128$.  
	We construct four groups of test problems based on \eqref{eq:gen-A},
	in each of which there is only one parameter varying while all others are fixed.
	Specifically, the problem parameter settings for $A$ are given as follows 
	(recall that $n$ is the number of rows,  $m$ is the number of columns, 
	$p$ is the number of principal components to be computed, and 
	$\xi$ determines the decay rate of singular values):
	\begin{enumerate}
		
		\item[(1)] $n = 1000 + 1000j$ for $j = 1, 2, 3, 4$, while $m = 128000$, $p = 20$, and $\xi = 1.01$;
		
		\item[(2)] $m = 128000 + 32000j$ for $j = 1, 2, 3, 4$, while $n = 2000$, $p = 10$, and $\xi = 1.01$;
		
		\item[(3)] $p = 10j$  for $j = 1, 2, 3, 4$, while $n = 1000$, $m = 128000$, and $\xi = 1.01$;
		
		\item[(4)]  $\xi =1+ 10^{-(1 + j)/2}$ for $j = 1, 2, 3, 4$, while $n = 1000$, $m = 256000$, and $p = 10$.	
		
	\end{enumerate}
	The numerical results for the above four test scenarios are depicted 
	in Figure \ref{fig:synthetic}, with two quantities, wall-clock time in seconds
	and number of iterations taken recorded for every experiment.
	The average scaled KKT violation and relative error of every experiment 
	are tabulated in Table \ref{tb:test_synthetic}. 
	It should be evident from these numerical results that 
	FAPS clearly outperforms {SSI and LocalPower
	in terms of iteration numbers.}
	% with the latter  (i.e., the speed of convergence) being the determining factor.	
	
	\begin{figure}[ht!]
		\centering
		\subfigure[Varying $n$]{
			\label{subfig:bar_n_time}
			\includegraphics[width=0.4\textwidth]{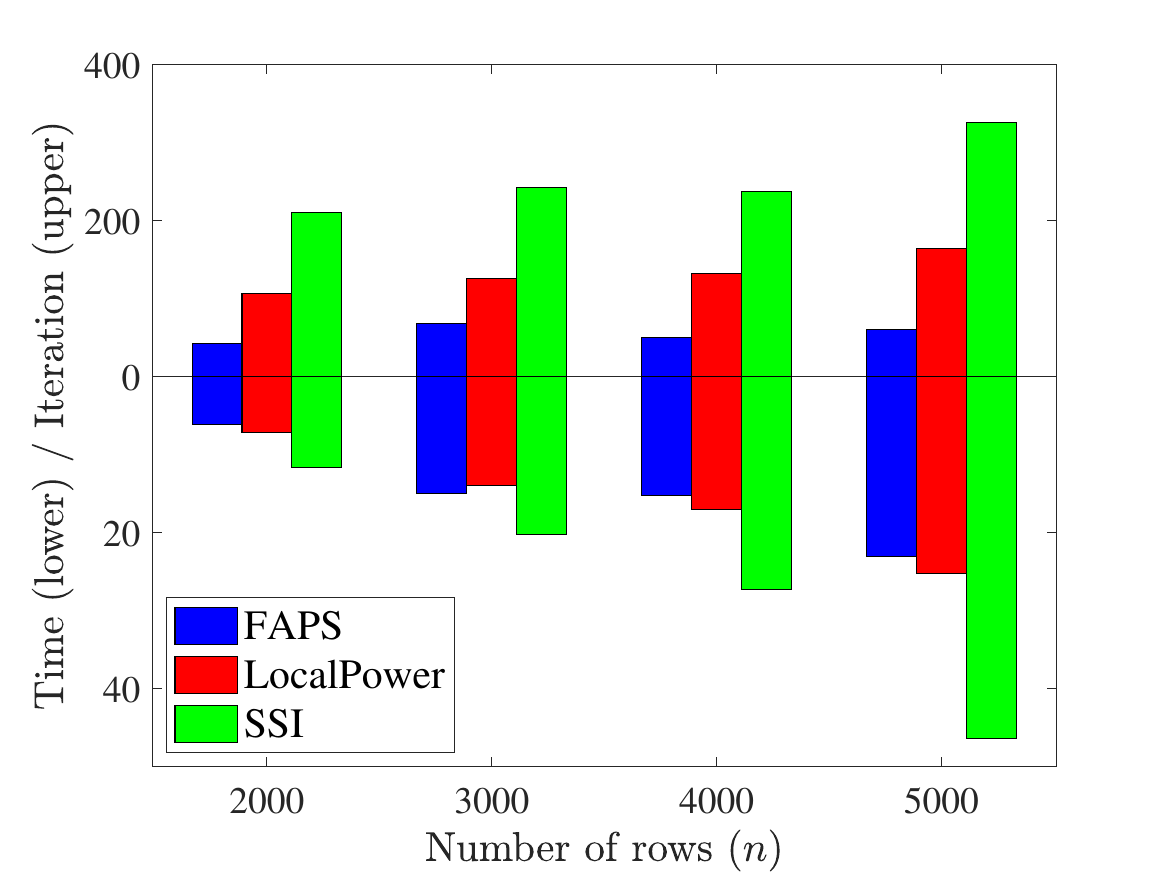}
		}
		\subfigure[Varying $m$]{
			\label{subfig:bar_m_time}
			\includegraphics[width=0.4\textwidth]{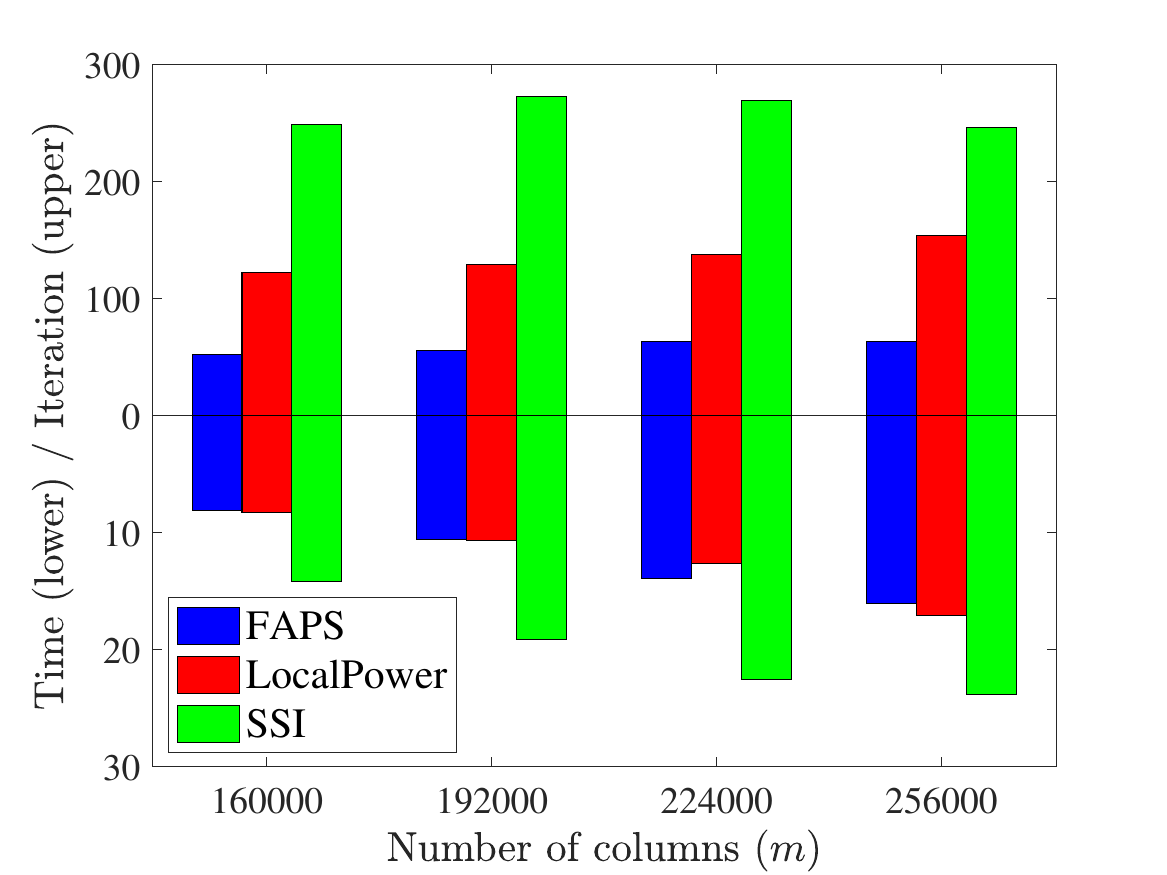}
		}\\
		\subfigure[Varying $p$]{
			\label{subfig:bar_p_time}
			\includegraphics[width=0.4\textwidth]{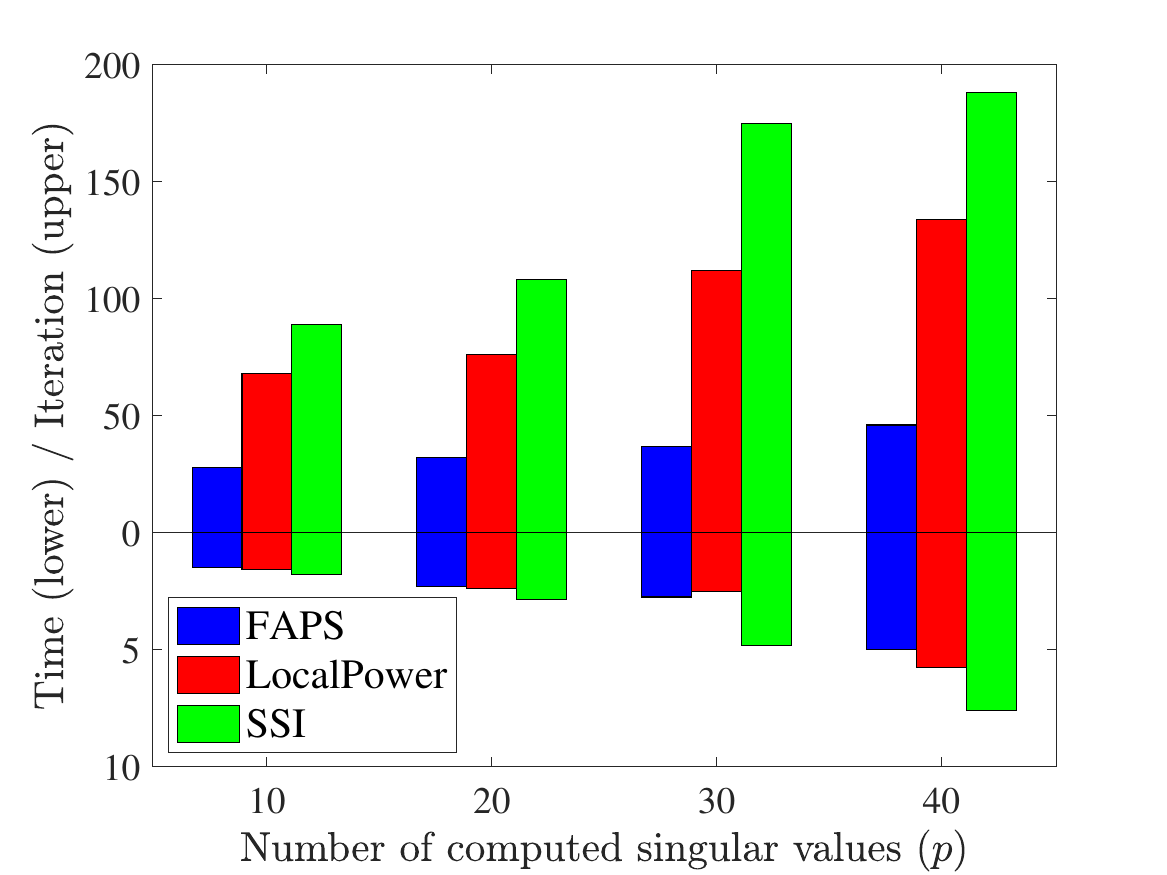}
		}
		\subfigure[Varying $\xi$]{
			\label{subfig:bar_xi_time}
			\includegraphics[width=0.4\textwidth]{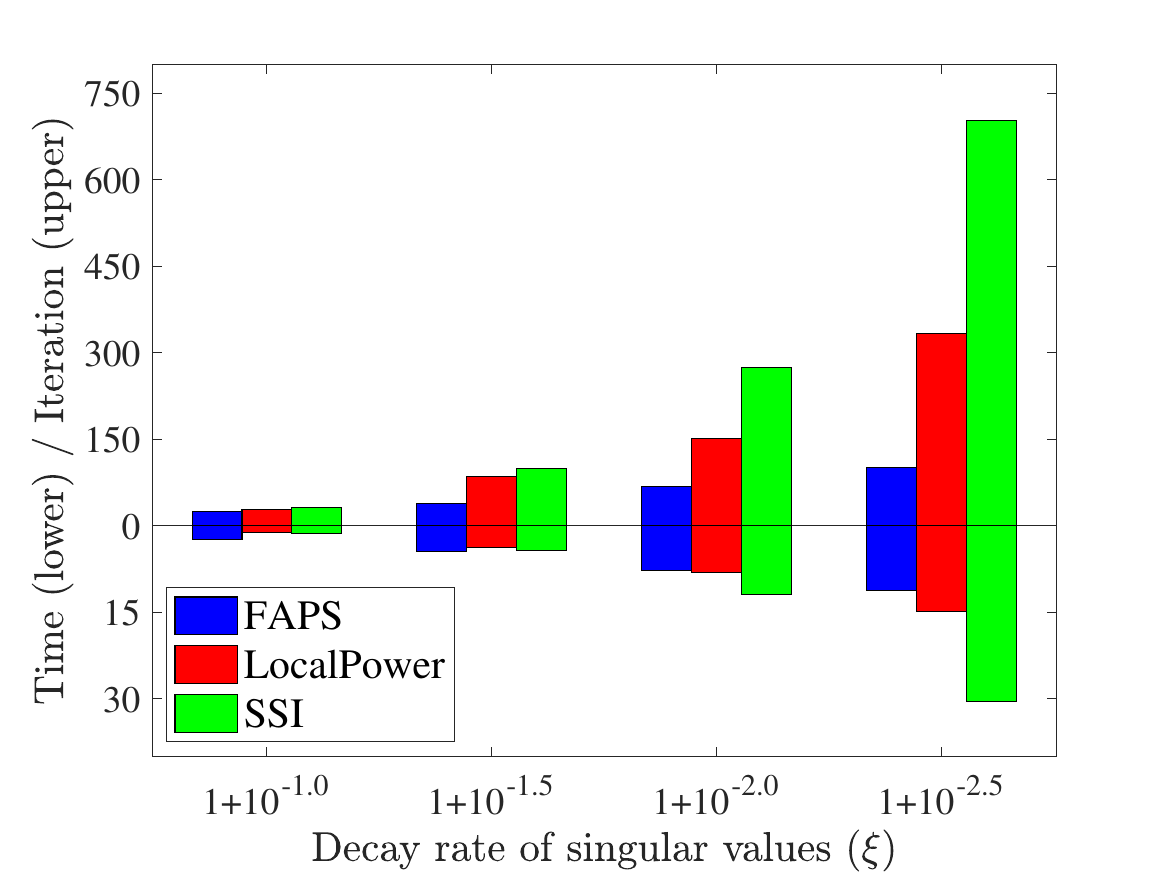}
		}
		
		\caption{Comparison of SSI, LocalPower, and FAPS on synthetic data.}
		\label{fig:synthetic}
	\end{figure}
	
	\begin{table}[ht!]
		\centering
		%\footnotesize
		\begin{tabular}{c|c|c|c|c|c|c} 
			\toprule 
			& \multicolumn{3}{c|}{Average scaled KKT violation} & \multicolumn{3}{c}{Average relative error}\\
			\midrule
			& SSI  & LocalPower & FAPS & SSI  & LocalPower  & FAPS \\
			\midrule
			Varying $n$ & 1.94e-06  & 1.89e-06 & 1.87e-06 & 1.06e-07 & 9.56e-08 &8.69e-08 \\
			\midrule
			Varying $m$ & 1.95e-06 & 1.92e-06 &1.90e-06 & 1.07e-07 & 1.02e-07 & 9.97e-08 \\
			\midrule
			Varying $p$ & 1.95e-06 & 1.91e-06 & 1.92e-06 & 8.30e-08 & 9.18e-08 &  1.15e-07 \\
			\midrule
			Varying $\xi$ & 4.09e-06  & 3.94e-06 & 3.87e-06 & 3.15e-07 & 3.08e-07 & 2.99e-07 \\
			\bottomrule 
		\end{tabular}
		\caption{Average errors of SSI, LocalPower and FAPS on synthetic data.}
		\label{tb:test_synthetic}
	\end{table}
	
	It is worth emphasizing that since these three algorithms incur more or less the same 
	amount of communication overhead per iteration, the total amount of information 
	exchanged is roughly proportional to the numbers of iterations.  Hence, the rapid
	convergence of FAPS (in terms of iteration number) translates into not only 
	computational but also communicational efficiency.  
	On the other hand, we caution that the advantage of FAPS may not always
	be as large as shown in our experiments when some parameter values go beyond the tested ranges .

\subsection{Comparison on Image Data Sets}
	
	\label{subsec:real}

	We next evaluate the performances of the three algorithms 
	on four image data sets popular in machine learning research.   
	The numbers of computed principal components and computing clients in use 
	are set to $p = 5$ and $d =16$, respectively.
	Numerical results from this experiment are given in Figure \ref{fig:image} and Table \ref{tb:test_mnist}.
	%Again, in terms of both wall-clock time and number of iterations taken but especially the latter,
	Again, in terms of the number of iterations taken,
	FAPS always dominates SSI and LocalPower.
	These results indicate that the observed superior performance of FAPS is not just limited 
	to synthetic matrices.
	
	\begin{figure}[ht!]
		\centering
		\subfigure[Total number of iterations]{
			\label{subfig:bar_image_iter}
			\includegraphics[width=0.4\textwidth]{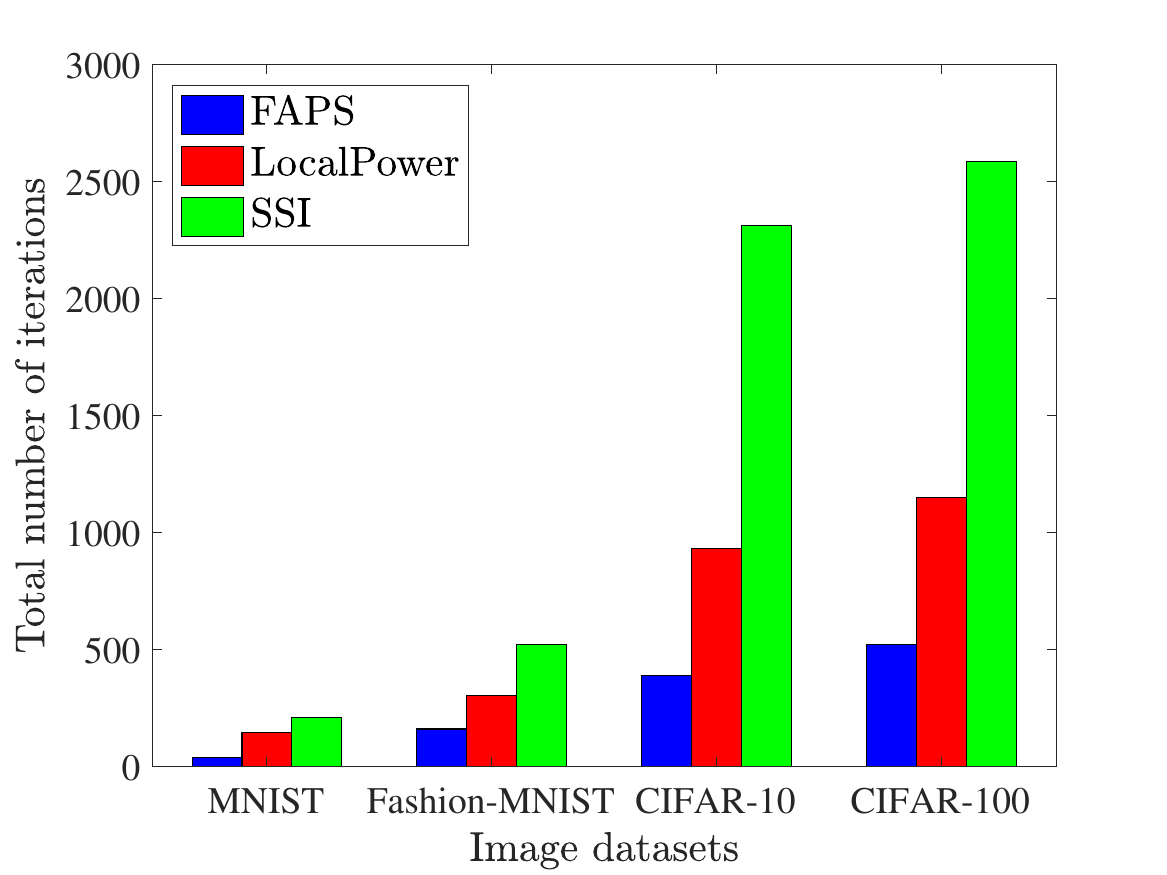}
		}
		\hspace{2mm}
		\subfigure[Wall-clock time in seconds]{
			\label{subfig:bar_image_time}
			\includegraphics[width=0.4\textwidth]{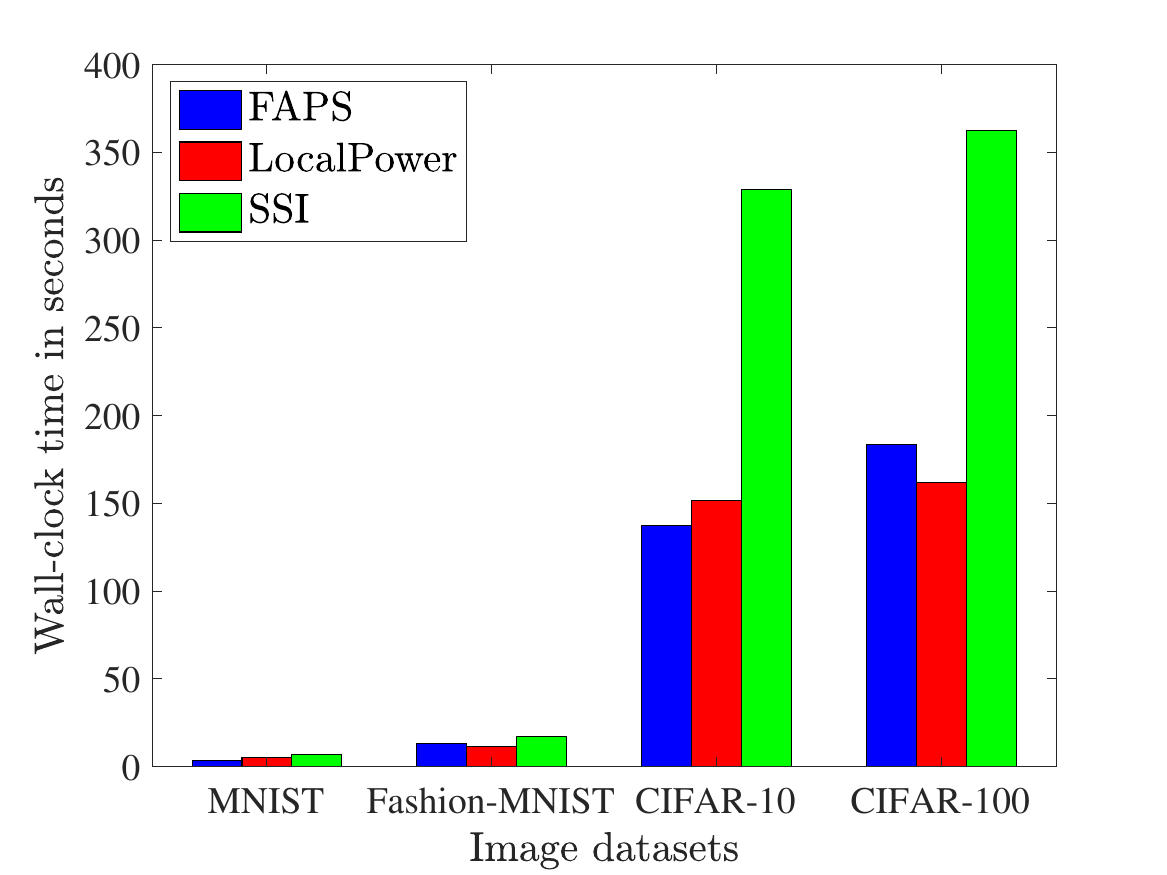}
		}
		\caption{Comparison of SSI, LocalPower, and FAPS on four image data sets.}
		\label{fig:image}
	\end{figure}
	
	\begin{table}[ht!]
		\centering
		%\footnotesize
		\begin{tabular}{c|c|c|c|c|c} 
			\toprule
			\multicolumn{3}{c|}{Average scaled KKT violation} & 
			\multicolumn{3}{c}{Average relative error}\\
			\midrule
			SSI  & LocalPower & FAPS & SSI & LocalPower & FAPS\\
			\midrule
			4.34e-06 & 4.59e-06 & 4.42e-06 & 5.80e-08 & 6.27e-08 & 5.06e-08 \\
			\bottomrule 
		\end{tabular}
		\caption{Average errors of SSI, LocalPower and FAPS on four image data sets.}
		\label{tb:test_mnist} %\small
	\end{table}

\subsection{Empirical Convergence Rate}

\label{subsec:rate}

	In this subsection, we take a closer look at the empirical convergence rate of FAPS.
	The behaviors of SSI and LocalPower are also studied in comparison.
	The test matrix $A$ is randomly generated by a similar manner 
	as \eqref{eq:gen-A} with $n = 1000$ and $m = 160000$
	except that the diagonal entries of $\Sigma$ satisfy the following arithmetic distribution.
	\begin{equation*}
		\Sigma_{ii} = 1 - \dfrac{i - 1}{n - 1} \dkh{1 - \dfrac{1}{\kappa}}, \quad \iin,
	\end{equation*}
	where $\kappa > 0$ is the condition number of $A$.
	In this experiment, we choose three different values of $\kappa$ to test, 
	including $10$, $100$, and $1000$.
	The numbers of computed principal components and computing clients in use 
	are set to $p = 10$ and $d =128$, respectively.
	
	As illustrated in Figure \ref{fig:rate}, all three algorithms appear to converge at linear rates 
	where the rate of FAPS is the fastest, followed by that of LocalPower.
	We observe that the larger $\kappa$ is, the faster the singular values decay, 
	and the easier the problem tends to be.
	On the tested instances, however, the advantage of FAPS relative to the other two 
	appears insensitive to the change of the decay rate in singular values.

	\begin{figure}[ht!]
		\centering
		\subfigure[$\kappa = 10$]{
			\label{subfig:kappa_10}
			\includegraphics[width=0.3\textwidth]{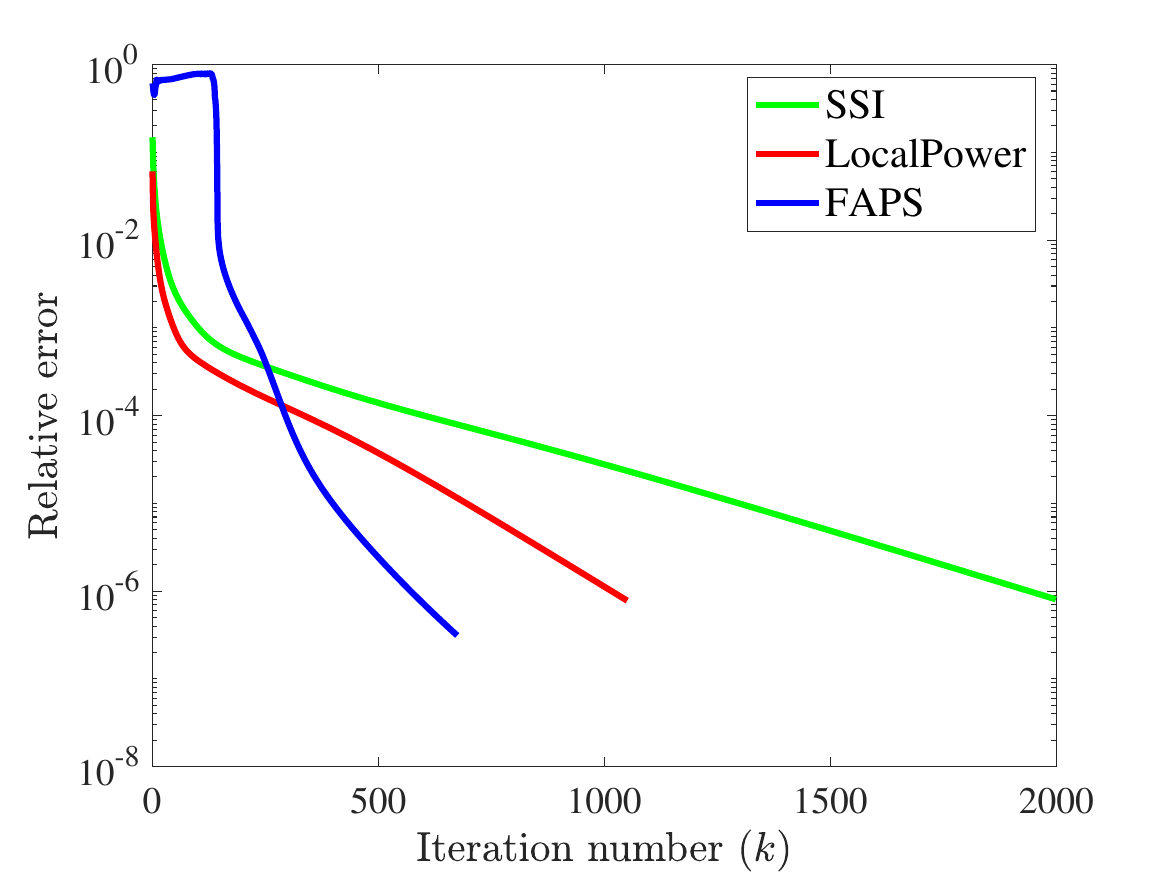}
		}
		\subfigure[$\kappa = 100$]{
			\label{subfig:kappa_100}
			\includegraphics[width=0.3\textwidth]{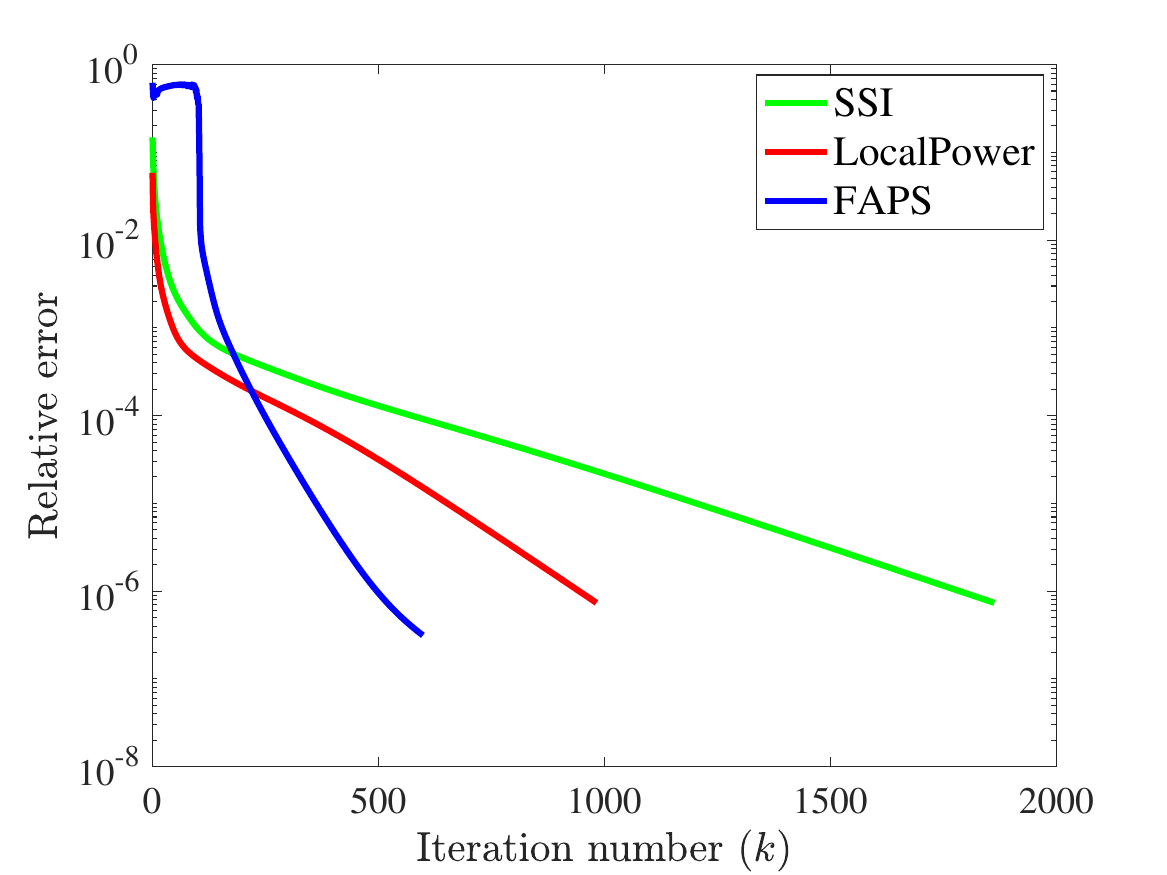}
		}
		\subfigure[$\kappa = 1000$]{
			\label{subfig:kappa_1000}
			\includegraphics[width=0.3\textwidth]{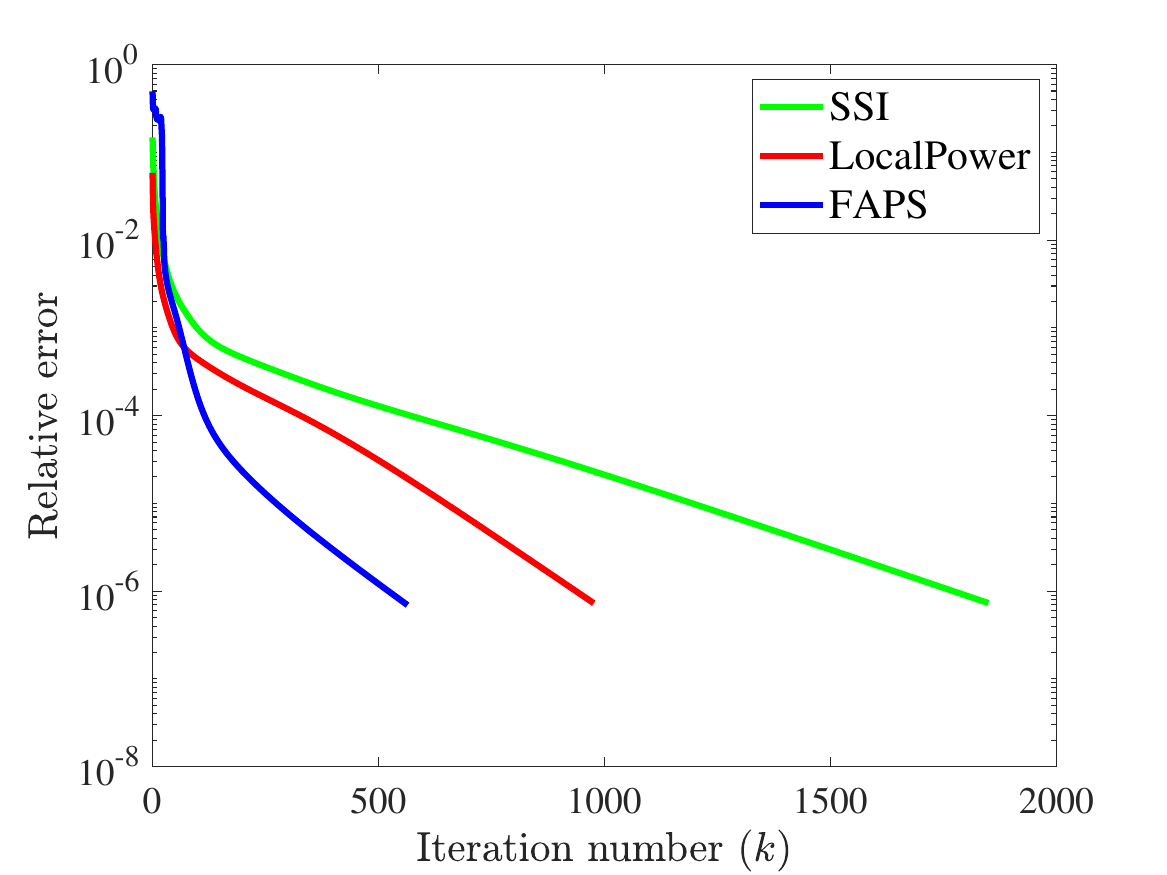}
		}
		
		\caption{Comparison of empirical convergence rates.}
		\label{fig:rate}
	\end{figure}

\subsection{Comparison on Unevenly Distributed Data}

\label{subsec:uneven}

	In the previous numerical experiments,
	the data set in test is uniformly distributed into $d$ clients,
	that is, $m_1 = m_2  = \dotsb = m_d$.
	Now we consider the scenario that 
	the whole data set is unevenly distributed across the network.
	In the following experiment, a synthetic matrix $A$, 
	generated by \eqref{eq:gen-A} with $n = 1000$, $m = 36000$, and
	$\xi = 1.01$, is tested with $p =10$ and $d = 8$.  
	We split the matrix $A$ into $d$ clients such that
	$m_i = 1000i$ for $\iid$.
	The corresponding numerical results are presented in Table \ref{tb:uneven},
	which demonstrate that FAPS attains better performances than SSI and LocalPower.
	In particular, we note that the number of iterations (i.e., the rounds of communications) 
	required by FAPS is about one-third of that by LocalPower, 
	even though the wall-clock times consumed by the two are close, 
	implying that FAPS could be far more communication-efficient 
	whenever the cost of communication is high.

	\begin{table}[ht!]
		\centering
		%\footnotesize
		\begin{tabular}{c|c|c|c|c} 
			\toprule 
			Algorithm & Iteration & Wall-clock~\!(s) & Scaled KKT violation & Relative error \\
			\midrule
			SSI                 &  337  &  32.73  &  1.94e-06  &  1.06e-07 \\
			\midrule
			LocalPower  &  164  &  16.52  &  1.96e-06  &  1.08e-07 \\
			\midrule
			FAPS             &  55  &  14.93  &  1.80e-06  &  7.67e-08 \\
			\bottomrule 
		\end{tabular} 
		\caption{Comparison of SSI, LocalPower, and FAPS on unevenly distributed data.}
		\label{tb:uneven} %\small
	\end{table}

\section{Conclusions}

	In addition to the traditional performance criteria of time and space efficiency,
	algorithms for computing PCA on large-scale data sets in federated environments 
	need to consider reducing communication overhead and, in many modern applications, 
	preserving data privacy.  Towards achieving these goals, we propose a new
	model and an ADMM-like algorithmic framework, called FAPS, that
	seeks consensus on a subspace rather than on a matrix variable.   
	
	This FAPS framework generates local subproblems that are 
	standard symmetric eigenvalue problems to which well-developed solvers readily apply.  
	From the viewpoint of the global variable,
	FAPS can be interpreted as an enhanced version of the federated SSI that can
	reduce the communication costs and safeguard local data privacy.  In addition,
	FAPS is equipped with two key algorithmic features:
	(i) multipliers are represented by a closed-form, low-rank formula; and
	(ii) solution accuracies for subproblems are appropriately controlled at low levels.
	
	Most existing theoretical works on convergence of ADMM algorithms for solving 
	non-convex optimization problems impose restrictive assumptions on iterates 
	or multipliers.  In our specific case for FAPS to solves a non-convex optimization 
	model with coupled nonlinear constraints, we have derive global convergence and 
	worst-case complexity results only under mild conditions on the choices of algorithm 
	parameters. 
	%The linear convergence rate of FAPS demands a further investigation.
	
	We have conducted comprehensive numerical experiments to compare FAPS with
	two competing algorithms for Federated PCA.  The test results are strongly in favor of FAPS.  
	Most notably, the number of iterations required by FAPS is significantly 
	fewer than that required by others up to one or two orders of magnitudes.   
	We believe that this fast empirical convergence rate is made possible by our
	subspace-splitting idea, producing greatly relaxed feasibility restrictions 
	relative to the classic variable-splitting strategy.
		
	Finally, we mention that there still remains a range of issues, theoretical or practical, 
	to be further studied in order to fully understand the behavior and realize the potential 
	of FAPS and its variants.
	We also note that the projection splitting idea can be generalized to a wider class of problems.

\section*{Acknowledgments}

	The work of the first author was supported by 
	the National Natural Science Foundation of China (No. 11971466 and 11991020).
	The work of the second author was supported in part by 
	the National Natural Science Foundation of China (No. 12125108, 11991021, and 12288201), 
	Key Research Program of Frontier Sciences, 
	Chinese Academy of Sciences (No. ZDBS-LY-7022), 
	the National Center for Mathematics and Interdisciplinary Sciences, 
	Chinese Academy of Sciences and the Youth Innovation Promotion Association, 
	Chinese Academy of Sciences.
	The work of the third author was supported in part by 
	the Shenzhen Science and Technology Program 
	(No. GXWD20201231105722002-20200901175001001).

\appendix
%\appendixpage
%\renewcommand{\appendixname}{Appendix~\Alph{section}}	

\section*{Appendix A. Proof of the Existence of Low-rank Multipliers} \label{apx:kkt-ps}
%\section{{Proof of the Existence of Low-rank Multipliers}}\label{apx:kkt-ps}

	In this appendix, we prove Proposition \ref{prop:kkt-multipliers} 
	to interpret the existence of low-rank multipliers 
	associated with the subspace constraints in \eqref{eq:opt-ps}.

	\begin{proof}[Proof of Proposition \ref{prop:kkt-multipliers}]
		We start with proving the ``only if" part, 
		and hence assume that $Z$ is a first-order stationary point of \eqref{eq:opt-trace}. 
		Let $\Theta = 0$, $\Gamma_i = - X_i\zz A_iA_i\zz X_i$, 
		and $\Lambda_i = -\Pv{X_i} A_iA_i\zz X_iX_i\zz - X_iX_i\zz A_iA_i\zz \Pv{X_i}$ with $\iid$. 
		Then matrices $\Theta$, $\Gamma_i$ and $\Lambda_i$ are symmetric and $\rank \dkh{\Lambda_i} \leq 2p$. 
		And it can be readily verified that 
		\begin{equation*}
			A_iA_i\zz X_i + X_i \Gamma_i + \Lambda_i X_i 
			= \Pv{X_i} A_iA_i\zz X_i - \Pv{X_i} A_iA_i\zz X_i
			= 0,
			\;\; \iid.
		\end{equation*}
		Moreover, it follows from the fact $X_i X_i\zz = Z Z\zz$ and stationarity of $Z$ that
		\begin{equation*}
			\sumiid\Lambda_i Z - Z \Theta
			= \sumiid\dkh{ -\Pv{X_i} A_iA_i\zz X_iX_i\zz - X_iX_i\zz A_iA_i\zz \Pv{X_i} } Z %\\
			%={} & \sumiid \dkh{ \dkh{ ZZ\zz - I_n } A_iA_i\zz ZZ\zz + ZZ\zz %A_iA_i\zz \dkh{ ZZ\zz - I_p } } Z = 0. %\\
			=  -\Pv{Z} AA\zz Z = 0.
			%%lx-delte %
		\end{equation*}
		%The other three equalities in \cref{eq:kkt-multipliers} hold straightforwardly. 
		Hence, $(\{X_i\},Z)$ satisfies the condition \eqref{eq:kkt-multipliers} 
		under the specific combination of $\Theta$, $\Gamma_i$, and $\Lambda_i$.
		
		Now we prove the ``if" part and assume that there exist symmetric matrices $\Theta$, $\Gamma_i$, and $\Lambda_i$ 
		such that the feasible point $(\{X_i\},Z)$ satisfies the condition \eqref{eq:kkt-multipliers}. 
		By virtue of \eqref{eq:kkt-multipliers}, 
		we obtain $A_iA_i\zz X_i = -X_i \Gamma_i - \Lambda_i X_i$, and hence it holds that
		\begin{equation*}
			 \sumiid \Pv{X_i} A_iA_i\zz X_i X_i\zz 
			=  - \sumiid \Pv{X_i} \dkh{ X_i\Gamma_i + \Lambda_i X_i } X_i\zz 
			%& = \sumiid \revise{-\Pv{X_i}} \Lambda_i X_iX_i\zz
			= - \Pv{Z} \dkh{\sumiid \Lambda_i Z} Z\zz =0,
			%= \dkh{ ZZ\zz - I_n } \dkh{ \sumiid\Lambda_i Z } Z\zz \\
			%= {} & \dkh{ ZZ\zz - I_n } Z\Theta Z\zz = 0,
			%%lx-delte %
		%	=  \revise{-\Pv{Z}} Z\Theta Z\zz = 0,
		\end{equation*}
		%where the last equality follows from \cref{eq:kkt-opt-z}.
		where the second equality follows from the fact $X_i X_i\zz = Z Z\zz$,
		and the third equality follows from \eqref{eq:kkt-multipliers}.
		On the other hand, we have 
		\begin{equation*}
			\sumiid \Pv{X_i} A_iA_i\zz X_i X_i\zz 
			=  \sumiid \Pv{Z} A_iA_i\zz Z Z\zz 
			=  \Pv{Z} AA\zz Z Z\zz.
		\end{equation*}
		Combining the above two relationships, we arrive at 
		$\Pv{Z} AA\zz Z Z\zz = 0$, which further implies
		%Then it follows from the orthogonality of $Z$ that 
		$\Pv{Z} AA\zz Z = 0$. 
		Therefore, $Z$ is a first-order stationary point of \eqref{eq:opt-trace}.
		We complete the proof.
	\end{proof}

\section*{Appendix B. Proof of the Global Convergence} \label{apx:global}
%\section{Proof of the Global Convergence} \label{apx:global}

	In this appendix, we prove Theorem \ref{thm:global} 
	to establish the global convergence of Algorithm \ref{alg:FAPS}.
	To begin with, we give an explicit expression of the constant $\omega_i > 0$ 
	in Assumption \ref{asp:beta-svd} as follows:
	\begin{equation*}
		\omega_i = \max 
		\hkh{ 
			c_1^{\prime}, \;
			\dfrac{12 \rho d \sqrt{p}}{c_1 \underline{\sigma}^2}, \;
			\dfrac{4\sqrt{2} \dkh{1 + \sqrt{2 \rho d}}}{\underline{\sigma} - 2\sqrt{\rho d} \delta_i}, \;
			16\rho d \sqrt{p}, \; 
			\dfrac{ 4 (1 + \sqrt{2\rho d}) }{ c_1 \underline{\sigma}^2 \rho d}
		},
		\;\; \iid.
	\end{equation*}
	In addition, it is clear that $\norm{A}\ff \geq \norm{A_i}\ff \geq \norm{A_i}_2$.
%	Then we provide a sketch of our proof.
%	Suppose $\{ \{X_i^{(k)}\}, Z^{(k)} \}$ is the iterate sequence generated by \cref{alg:FAPS}.
%	The main steps of the proof include:
%	\begin{enumerate}
%		
%		\item \revise{Any iterate $( \{X_i^{(k)}\}, Z^{(k)} )$ satisfies
%		$\ditk \leq 1 / (\rho d)$, $\iid$,
%		where $\dik$ is defined in \cref{eq:Dik};}
%		%and $\underline{\sigma}$  is a uniformly lower bound of the smallest singular values of the matrices $(X_i^{(k)})\zz Z^{(k)}(\iid)$;
%		
%		\item The sequence $\{ \cL ( \{X_i^{(k)}\}, Z^{(k)}, \{\Lambda_i^{(k)}\} ) \}$ is bounded from below and monotonically non-increasing, and hence is convergent;
%		
%		\item The sequence $\{ Z^{(k)} \}$ has at least one accumulation point,
%		and any accumulation point is a first-order stationary point of \cref{eq:opt-trace}.
%		
%	\end{enumerate}

%	Next we verify all the assertions 
%	in the above sketch by establishing a few lemmas and corollaries.
	Next, in order to prove Theorem \ref{thm:global},
	we establish a few lemmas and corollaries to make preparations.
	% % % lx-revise % % %
	%In the proofs, we omit the superscript $(k)$ for $X_i^{(k)}$, $Z^{(k)}$, and $\Lambda_i^{(k)}$ to save space with a slight abuse of notations.
	In their proofs, we omit the superscript $(k)$ to save space with a slight abuse of notations, and use the superscript $+$ to take the place of $(k+1)$.
	%Moreover, we use $X_i^+$, $Z^+$, and $\Lambda_i^+$ to denote the next iterates $X_i^{(k+1)}$, $Z^{(k+1)}$, and $\Lambda_i^{(k+1)}$, respectively.
	
	\begin{lemma}
		\label{le:dist-bridge}
		Suppose Assumption \ref{asp:beta-svd} holds 
		and $( \{X_i^{(k)}\}, Z^{(k)} )$ is the $k$-th iterate generated by Algorithm \ref{alg:FAPS} 
		and satisfies that $\dik \leq \sqrt{1 / (\rho d)}$, $\iid$.
		%$\sigma_{\min} ( (X_i^{(k)})\zz Z^{(k)} ) \geq \underline{\sigma}$.
		Then it holds that
		\begin{equation*}
			h_i^{(k)} (X_i^{(k)}) - h_i^{(k)} (X_i^{(k+1)})
			\geq \dfrac{1}{4}  c_1 \underline{\sigma}^2 \beta_i \ditk, \quad \iid.
			%\norm{\dkh{I_n-X_i^{(k)}(X_i^{(k)})\zz}Z^{(k)}(Z^{(k)})\zz X_i^{(k)}}\fs \geq \dfrac{1}{2} \underline{\sigma}^2 \norm{X_i^{(k)}(X_i^{(k)})\zz-Z^{(k)}(Z^{(k)})\zz}\fs.
		\end{equation*}
		%where $D_i^{(k)}$ is defined in \cref{eq:Dik}.
	\end{lemma}
	
	\begin{proof}
		It follows from Assumption \ref{asp:beta-svd} that 
		$\beta_i > c_1^{\prime} \norm{A}\fs \geq c_1^{\prime} \norm{A_i}_2^2$, 
		which together with \eqref{eq:ps-sub-x-con-1} yields that
		\begin{equation}\label{eq:des-x}
			h_i ( X_i ) - h_i ( X_i^+ ) \geq \dfrac{c_1}{2\beta_i} \norm{ \Pv{X_i} H_i X_i }\fs.
		\end{equation}
		According to the definition of $H_i$ and $\Lambda_i$, we have
		\begin{equation}\label{eq:rgrad-h}
			%\dkh{ I_n - X_i X_i\zz } H_i^{(k)} X_i
			\Pv{X_i} H_i X_i
			= \Pv{X_i} \dkh{ A_iA_i\zz + \Lambda_i+ \beta_i Z Z\zz } X_i %\\
			%= {} & \dkh{ I_n - X_i X_i\zz } \dkh{ A_iA_i\zz X_i + \dkh{ X_i 
			%X_i\zz	- I_n } A_iA_i\zz X_i + \beta_i Z Z\zz X_i } \\
			= \beta_i \Pv{X_i} Z Z\zz X_i.
		\end{equation}
		Suppose $\hat{\sigma}_1, \dotsc, \hat{\sigma}_p$ are the singular values of $X_i\zz Z$.
		It is clear that $0 \leq \hat{\sigma}_i \leq 1$ 
		and $\dit = 2\sum_{j=1}^p \dkh{ 1 - \hat{\sigma}_j^2 }$ for any $\iid$.
		%due to the orthogonality of $X_i$ and $Z$.
		By simple calculations, we have
		\begin{equation*}
			\norm{ \Pv{X_i} Z Z\zz X_i }\fs 
			= \tr \dkh{ X_i\zz Z Z\zz X_i } - \tr \dkh{ \dkh{ X_i\zz Z Z\zz X_i}^2 }
			=  \sum\limits_{j=1}^p \hat{\sigma}_j^2 \dkh{ 1 - \hat{\sigma}_j^2 }.
		\end{equation*}
		Moreover, it follows from $\dik \leq \sqrt{1 / (\rho d)}$ 
		that $\sigma_{\min}( X_i\zz Z ) \geq \underline{\sigma}$,
		which implies that
		\begin{equation*}%\label{eq:dist-svd}
			%\begin{aligned}
			\norm{ \Pv{X_i} Z Z\zz X_i }\fs 
			=  \sum\limits_{j=1}^p \hat{\sigma}_j^2 \dkh{ 1 - \hat{\sigma}_j^2 } 
			%\geq  \sum\limits_{j=1}^p \sigma_{\min}^2\dkh{ X_i\zz Z } \dkh{ 1 - \hat{\sigma}_j^2 } 
			\geq  \underline{\sigma}^2 \sum\limits_{j=1}^p \dkh{ 1 - \hat{\sigma}_j^2 }
			= \dfrac{1}{2} \underline{\sigma}^2 \dit.
			%\end{aligned}
		\end{equation*}
		%\begin{equation}
		%	\label{eq:xxzz}
			%\norm{ D_i^{(k)} }\fs = \norm{ X_i X_i\zz - Z Z\zz }\fs =
		%	\norm{ D_i }\fs := \norm{ X_i X_i\zz - Z Z\zz }\fs =
		%	2\sum\limits_{j=1}^p \dkh{ 1 - \hat{\sigma}_j^2 }.
		%\end{equation}
		%Combining the relationships \cref{eq:des-x} and %\cref{eq:xxzz}, 
		%we arrive at the following inequality:
		%\begin{equation*}
		%	$\norm{ \Pv{X_i} Z Z\zz X_i }\fs 
		%	\geq \dfrac{1}{2} \underline{\sigma}^2\dit$.
			%\norm{ D_i^{(k)} }\fs,
		%\end{equation*}
		This together with \eqref{eq:des-x} and \eqref{eq:rgrad-h} completes the proof.
	\end{proof}
	
	\begin{lemma}
		\label{le:dist-xk+1-zk}
		Suppose all the conditions in Lemma \ref{le:dist-bridge} hold.
		Then for any $\iid$, we have
		\begin{equation*}
			\dists{X_i^{(k+1)}}{Z^{(k)}} 
			\leq \dkh{ 1 - c_1 \underline{\sigma}^2 } \ditk 
			+ \dfrac{12}{\beta_i} \sqrt{p} \norm{A_i}\fs.
		\end{equation*}
	\end{lemma}
	
	\begin{proof}
		According to Lemma \ref{le:dist-bridge} and 
		%definitions of $h_i^{(k)}$ and $H_i^{(k)}$, we can obtain
		definitions of $h_i$ and $H_i$, we can acquire
		\begin{equation*}
			%\label{eq:add2}
			%\dfrac{1}{4} c_1 \underline{\sigma}^2 \norm{ D_i }\fs
			%\norm{ D_i^{(k)} }\fs
			 \dfrac{1}{2} \tr \dkh{ Z Z\zz \dsp{X_i^+}{X_i}}
			+ \dfrac{1}{2\beta_i} \tr \dkh{ \dkh{A_iA_i\zz + \Lambda_i } 
			\dsp{X_i^+}{X_i}} 
			\geq \dfrac{1}{4} c_1 \underline{\sigma}^2 \dit.
			%= \dfrac{1}{4} c_1 \underline{\sigma}^2 \norm{ X_i X_i\zz - Z Z\zz }\fs.			
		\end{equation*}
		By straightforward calculations, 
		we can further obtain the following two relationships
		\begin{align*}
			\tr \dkh{ \dkh{ A_iA_i\zz + \Lambda_i } \dsp{X_i^+}{X_i}}  
			 & \leq \norm{ A_iA_i\zz + \Lambda_i }\ff \dist{X_i^+}{X_i} 
			 \leq  6 \sqrt{p} \norm{A_i}\fs, \\
			\tr \dkh{ Z Z\zz \dsp{X_i^+}{X_i}}
			%= {} & \dfrac{1}{2} \norm{ X_i X_i\zz - Z Z\zz }\fs  - \dfrac{1}{2} 
			%\norm{ X_i^+ (X_i^+)\zz - Z Z\zz }\fs \\
			& = \dfrac{1}{2}  \dit %\norm{ D_i^{(k)} }\fs  
			- \dfrac{1}{2}  \dists{X_i^+}{Z}.
		\end{align*}
		Combining the above three relationships, we complete the proof.
	\end{proof}
	
	\begin{lemma}
		\label{le:sub-z}
		Suppose $\{ \{X_i^{(k)}\}, Z^{(k)} \}$ is the iterate sequence generated by Algorithm \ref{alg:FAPS}. 
		Then, for $\iid$ and $k \in \N$, the following inequality holds.
		\begin{equation}
			\label{eq:ps-sub-z-con}
			{q}^{(k)} (Z^{(k)}) - {q}^{(k)} (Z^{(k+1)}) 
			\geq c_2 \norm{ \Pv{Z^{(k)}} Q^{(k)} Z^{(k)} }\fs,
		\end{equation}
		 where $c_2 > 0$ is a constant dependent on the penalty parameters $\beta_i (\iid)$.
	\end{lemma}

	\begin{proof}
		This is a direct consequence of Lemma 3.5 in \citep{Liu2013}.
	\end{proof}

	\begin{lemma}
		\label{le:dist-zk}
		Suppose $\{ \{X_i^{(k)}\}, Z^{(k)} \}$ is the iterate sequence generated by Algorithm \ref{alg:FAPS}. 
		Then the inequality
		\begin{equation*}
			%\label{eq:dist-zk}
			\ditkp 
			%\norm{ D_i^{(k + 1)} }\fs 
			\leq \rho \sumjjd 
			\dists{X_j^{(k+1)}}{Z^{(k)}}
			%\norm{X_j^{(k+1)}(X_j^{(k+1)})\zz - Z^{(k)}(Z^{(k)})\zz}\fs
			+ \dfrac{8\sqrt{p}}{\beta_i} \norm{A}\fs
		\end{equation*}
		holds for $\iid$ and $k \in \N$.
		%where $\rho \geq 1$ is a constant defined in Assumption \cref{asp:step}.
	\end{lemma}
	
	\begin{proof}
		The inequality \eqref{eq:ps-sub-z-con} directly results in the
		%relationship ${q}^{(k)} (Z) - {q}^{(k)} (Z^+) \geq 0$.
		relationship ${q} (Z) - {q} (Z^+) \geq 0$,
		%Recall that $\bar{q}^{(k)}$ is the objective function of subproblem \cref{eq:ps-sub-z}.
		which yields that
		%According to the definition of $q^{(k)}$ and $Q^{(k)}$, it follows that
		\begin{equation*}
			0 \leq \sumjjd \beta_j \tr \dkh{  X_j^+ (X_j^+)\zz 
				\dsp{Z^+}{Z}}
			+ \sumjjd \tr \dkh{  \Lambda_j^+ \dsp{Z}{Z^+}}.
		\end{equation*}
		By straightforward calculations, we can deduce the following two 
		relationships
		\begin{align*}
			\tr \dkh{  \Lambda_j^+ \dsp{Z}{Z^+} }
			& \leq \norm{\Lambda_j^+}\ff \dist{Z}{Z^+}
			\leq 4\sqrt{p} \norm{A_j}\fs,\\
			\tr \dkh{  X_j^+ (X_j^+)\zz \dsp{Z^+}{Z} }
			& = \dfrac{1}{2} \dists{X_j^+}{Z}
			-  \dfrac{1}{2} \djps,
		\end{align*}
		which implies that 
		\begin{equation*}
			\sumjjd \beta_j \djps
			\leq \sumjjd \beta_j \dists{X_j^+}{Z} + 8\sqrt{p} \norm{A}\fs.
		\end{equation*}
		Now it can be readily verified that
		\begin{equation*}
				%\norm{ D_i^{(k + 1)} }\fs
				\dips
				%= {} & \norm{ X_i^+ (X_i^+)\zz - Z^+ (Z^+)\zz }\fs
				\leq \dfrac{1}{\beta_i} \sumjjd \beta_j \djps
				\leq \rho \sumjjd \dists{X_j^+}{Z}
				+ \dfrac{8\sqrt{p}}{\beta_i} \norm{A}\fs.
		\end{equation*}
		This completes the proof.
	\end{proof}
	
	\begin{lemma}
		\label{le:lip-phi}
		Let $\Phi_i(Y)= -\Pv{Y} A_iA_i\zz YY\zz - YY\zz A_iA_i\zz \Pv{Y} $ 
		for any $Y \in \stiefel$ and $\iid$.
		Then for any $Y_1 \in \stiefel$ and $Y_2 \in \stiefel$, it holds that
		\begin{equation*}
			%\label{eq:lip-phi}
			\norm{\Phi_i(Y_1) - \Phi_i(Y_2)}\ff \leq 4\norm{A_i}_2^2 \dist{Y_1}{Y_2}, \quad \iid.
		\end{equation*}
	\end{lemma}
	
	\begin{proof}
		This lemma directly follows from the triangular inequality.
		Hence, its proof is omitted.
	\end{proof}
	
	\begin{lemma}
		\label{le:svd-xz}
		Suppose Assumption \ref{asp:beta-svd} holds, 
		and $\{ \{X_i^{(k)}\}, Z^{(k)} \}$ is the iterate 
		sequence generated by Algorithm \ref{alg:FAPS}
		initiated from $( \{X_i^{(0)}\}, Z^{(0)} )$ satisfying \eqref{eq:initial}.
		Then for $k \in \N$, it holds that 
		\begin{equation}
			\label{eq:svd-xz}
			\ditk \leq \dfrac{1}{\rho d}, \quad \iid.
		\end{equation}
	\end{lemma}
	
	\begin{proof}
		We use mathematical induction to prove this lemma.
		The argument \eqref{eq:svd-xz} directly holds at 
		$\{\diz\}_{i = 1}^d$ resulting from \eqref{eq:initial}.
		Now, we assume the argument holds at 
		$\{\di\}_{i = 1}^d$,
		%$\{D_i^{(k)}\}_{i = 1}^d$, which infers that 
		%which infers that 
		%$\sigma_{\min} \dkh{ X_i\zz Z } \geq \underline{\sigma}$, 
		and investigate the situation at $\{\di^+\}_{i = 1}^d$.
		%the situation at $\{D_i^{(k + 1)}\}_{i = 1}^d$. 
		
		According to Assumption \ref{asp:beta-svd}, we have
		$\beta_i > {12 \rho d \sqrt{p} \norm{A_i}\fs}/{(c_1 \underline{\sigma}^2)}$.
		Without loss of generality, we assume that $c_1 \underline{\sigma}^2 < 1$.
		Combining Lemma \ref{le:dist-xk+1-zk} and \eqref{eq:svd-xz}, 
		we can derive that
		\begin{equation*}
			\dists{X_i^+}{Z}
			%\norm{ X_i^+ (X_i^+)\zz - Z Z\zz }\fs 
			\leq \dfrac{1 - c_1 \underline{\sigma}^2}{\rho d} + \dfrac{c_1 \underline{\sigma}^2}{\rho d}
			= \dfrac{1}{\rho d},
		\end{equation*}
		which infers that $\sigma_{\min} \dkh{ (X_i^+)\zz Z } \geq \underline{\sigma}$.
		Similar to the proof of Lemma \ref{le:dist-bridge}, we can deduce that
		\begin{equation}
			\label{eq:svd-xk+1-zk}
			\norm{ \Pv{X_i^+} Z Z\zz X_i^+ }\fs \geq \dfrac{ \underline{\sigma}^2 }{2} 
			\dists{X_i^+}{Z}.
			%\norm{ X_i^+ (X_i^+)\zz - Z Z\zz }\fs.
		\end{equation}
		Together with condition \eqref{eq:ps-sub-x-con-1} and equality \eqref{eq:rgrad-h}, we have
		\begin{equation*}
			%\norm{ \dkh{ I_n - X_i^+ (X_i^+)\zz } H_i^{(k)} X_i^+ }\ff 
			\norm{ \Pv{X_i^+} H_i X_i^+ }\ff 
			\leq \delta_i \beta_i \norm{ \Pv{X_i}Z Z\zz X_i }\ff 
			\leq \delta_i \beta_i \di.
		\end{equation*}
		On the other hand, it follows from the triangular inequality that
		\begin{equation*}
			\norm{  \Pv{X_i^+} H_i X_i^+ }\ff
			%\norm{ \dkh{ I_n - X_i^+ (X_i^+)\zz } H_i^{(k)} X_i^+ }\ff
			\geq   \norm{  \Pv{X_i^+} \dkh{ A_iA_i\zz 
					+ \Lambda_i^+ + \beta_i Z Z\zz } X_i^+ }\ff 
			- \norm{  \Pv{X_i^+}\dkh{ \Lambda_i^+ - \Lambda_i } X_i^+ }\ff.
		\end{equation*}
		It follows from the inequality \eqref{eq:svd-xk+1-zk} that
		\begin{equation*}
			\norm{ \Pv{X_i^+} \dkh{ A_iA_i\zz + \Lambda_i^+ + \beta_i Z Z\zz } X_i^+ }\ff = \beta_i \norm{ \Pv{X_i^+} Z Z\zz X_i^+ }\ff
			\geq \dfrac{\sqrt{2}}{2} \underline{\sigma} \beta_i 
			\dist{X_i^+}{Z}.
		\end{equation*}
		According to Lemma \ref{le:lip-phi}, we have
		\begin{equation*}
			\norm{ \Pv{X_i^+} \dkh{ \Lambda_i^+ - \Lambda_i } X_i^+ }\ff 
			%\leq \norm{ \Lambda_i^+ - \Lambda_i }\ff  
			 \leq 4 \norm{A_i}_2^2 \dist{X_i^+}{X_i}
			 \leq 4 \norm{A_i}_2^2 \dkh{ \dist{X_i^+}{Z} + \di }.
		\end{equation*}
		Combing the above four inequalities, we further obtain that
%		\begin{equation*}
%			%\label{eq:rgrad-h-xk+1-2}
%			%\norm{ \dkh{ I_n - X_i^+ (X_i^+)\zz } H_i^{(k)} X_i^+ }\ff 
%			\delta_i \beta_i \di
%			\geq \norm{ \Pv{X_i^+} H_i X_i^+ }\ff 
%			%& \geq \dfrac{\sqrt{2}}{2} \underline{\sigma} \beta_i \norm{ X_i^+ (X_i^+)\zz - Z Z\zz }\ff 
%			%- 4 \norm{A_i}_2^2 \norm{ X_i^+ (X_i^+)\zz - X_i X_i\zz }\ff \\
%			\geq ( \sqrt{2}\underline{\sigma} \beta_i / 2 - 4 \norm{A_i}_2^2 ) %\norm{ X_i^+ (X_i^+)\zz - Z Z\zz }\ff
%			\dist{X_i^+}{Z}
%			- 4 \norm{A_i}_2^2 \di, %\norm{ X_i X_i\zz - Z Z\zz }\ff,
%		\end{equation*}
%		which yields that
		\begin{equation*}
			( \sqrt{2}\underline{\sigma} \beta_i / 2 - 4 \norm{A_i}_2^2 ) 
			%\norm{ X_i^+ (X_i^+)\zz - Z Z\zz }\ff
			\dist{X_i^+}{Z}
			\leq ( \delta_i \beta_i + 4 \norm{A_i}_2^2 ) \di.
			%\norm{ X_i X_i\zz - Z Z\zz }\ff.
		\end{equation*}
		According to Assumption \ref{asp:beta-svd},
		we have $0 \leq \delta_i < \underline{\sigma} / \sqrt{4\rho d}$, and
		\begin{equation*}
			\beta_i 
			> 
			\dfrac{4\sqrt{2} \dkh{1 + \sqrt{2 \rho d}}}{\underline{\sigma} - 2\sqrt{\rho d} \delta_i}\norm{A_i}_2^2
			\geq 
			\dfrac{4\sqrt{2}}{\underline{\sigma}}\norm{A_i}_2^2.
		\end{equation*}
		Thus, we arrive at 
		\begin{equation}
			\label{eq:dist-xxzz}
			%\norm{ X_i^+ (X_i^+)\zz - Z Z\zz }\ff 
			\dist{X_i^+}{Z}
			\leq  \dfrac{ 2 ( \delta_i \beta_i + 4 \norm{A_i}_2^2 ) }{ \sqrt{2} \underline{\sigma} \beta_i - 8 \norm{A_i}_2^2} \di
			%\norm{ X_i X_i\zz - Z Z\zz }\ff 
			\leq \sqrt{ \dfrac{1}{2 \rho d} } \di, \quad \iid.
			%\norm{D_i}\ff. %\norm{ D_i^{(k)} }\ff.
		\end{equation}
		Again, we have $\beta_i > 16\rho d \sqrt{p} \norm{A}\fs$
		according to Assumption \ref{asp:beta-svd}.
		Combing Lemma \ref{le:dist-zk} and \eqref{eq:svd-xz}, we further acquire that
		\begin{equation*}
			%\label{eq:ub-xxzz}
			%\norm{ D_i^{(k + 1)}}\fs
			\dips
			\leq \rho \sumjjd \dists{X_j^+}{Z}
			+ \dfrac{8\sqrt{p}}{\beta_i} \norm{A}\fs
			\leq \dfrac{1}{2 d} \sumjjd 
			%\norm{ D_i^{(k)} }\ff + \dfrac{1}{2 \rho d}
			\mathbf{d}_j^2 + \dfrac{1}{2 \rho d}
			%\leq \dfrac{1}{2 \rho d} + \dfrac{1}{2\rho d}
			\leq \dfrac{1}{\rho d},
		\end{equation*}
		which completes the proof.
	\end{proof}
	
	\begin{corollary}
		\label{cor:des-xy}
		Suppose all the conditions in Lemma \ref{le:svd-xz} hold. 
		Then for any $k \in \N$, there holds
		\begin{equation*}
			%\label{eq:dist-xk}
			\cL(\{X_i^{(k)}\}, Z^{(k)}, \{\Lambda_i^{(k)}\}) - 
			\cL(\{X_i^{(k+1)}\}, Z^{(k)}, \{\Lambda_i^{(k)}\})
			\geq \dfrac{1}{4} c_1 \underline{\sigma}^2
			\sumiid \beta_i \ditk.
		\end{equation*} 
	\end{corollary}
	
	\begin{proof}
		%Recall that $D_i^{(k)}$ is defined in \cref{eq:Dik}.
		This corollary directly follows from Lemmas \ref{le:dist-bridge} and \ref{le:svd-xz}.
	\end{proof}
	
	\begin{corollary}
		\label{cor:des-lambda}
		Suppose all the conditions in Lemma \ref{le:svd-xz} hold. 
		Then for any $k \in \N$, it holds that {\small
		\begin{equation*}
		 \cL(\{X_i^{(k+1)}\}, Z^{(k)}, \{\Lambda_i^{(k)}\}) - 
		\cL(\{X_i^{(k+1)}\}, Z^{(k)}, \{\Lambda_i^{(k+1)}\}) 
		\geq - \dfrac{1 + \sqrt{2 \rho d}}{\rho d} \sumiid \norm{A_i}_2^2 \ditk.
			%\norm{ D_i^{(k)} }\fs.
		\end{equation*}}
	\end{corollary}
	
	\begin{proof}
		According to the Cauchy-Schwarz inequality, we can deduce that
		\begin{align*}
			& \abs{ \jkh{ \Lambda_i^+ - \Lambda_i, \dsp{X_i^+}{Z}  }} 
			= \abs{ \jkh{ \Phi_i(X_i^+) - \Phi_i(X_i),\dsp{X_i^+}{Z}  } } \\
			& \leq \norm{ \Phi_i(X_i^+) - \Phi_i(X_i) }\ff 
			\dist{X_i^+}{Z}
			%\norm{ X_i^+ (X_i^+)\zz - Z Z\zz}\ff
			\leq \sqrt{ \dfrac{8}{\rho d} } \norm{A_i}_2^2 \dist{X_i^+}{X_i} \di,
			%\norm{ D_i^{(k)} }\ff,
		\end{align*}
		where the last inequality follows from Lemma \ref{le:lip-phi} and \eqref{eq:dist-xxzz}.
		Moreover, we have
		\begin{equation*}
			\dist{X_i^+}{X_i}
			 \leq \dist{X_i^+}{Z} + \di 
			 \leq \dfrac{1 + \sqrt{2\rho d}}{\sqrt{2 \rho d}} \di,
			 %\norm{ D_i }\ff, %\norm{ D_i^{(k)} },
		\end{equation*}
		which  further yields that
		\begin{equation*}
			\jkh{ \Lambda_i^+ - \Lambda_i, \dsp{X_i^+}{Z} } 
			\geq - \dfrac{2 (1 + \sqrt{2 \rho d})}{\rho d} \norm{A_i}_2^2 \dit.
		\end{equation*}
		Combing the fact that
		\begin{equation*}
			\cL_i (X_i^+, Z, \Lambda_i) - \cL_i (X_i^+, Z, \Lambda_i^+)
			= \dfrac{1}{2} \jkh{ \Lambda_i^+ - \Lambda_i, \dsp{X_i^+}{Z} },
			%\norm{ D_i }\fs,
			%\norm{ D_i^{(k)} }\fs.
		\end{equation*}
		we complete the proof.
	\end{proof}
	
	\begin{corollary}
		\label{cor:des-z}
		Suppose $\{ \{X_i^{(k)}\}, Z^{(k)} \}$ is the iterate sequence generated by Algorithm \ref{alg:FAPS}.
		Let
		$
		%\begin{equation*}
			\bar{Q}^{(k)} = \sum_{i = 1}^d \dkh{\beta_i X_i^{(k+1)}(X_i^{(k+1)})\zz 
			+ \Phi_i (Z^{(k)}) - \Phi_i (X_i^{(k+1)}) },
		%\end{equation*}
		$
		and
		%\begin{equation*}
		$	
			G^{(k)} = \Pv{Z^{(k)}} AA\zz Z^{(k)} + \Pv{Z^{(k)}} \bar{Q}^{(k)} Z^{(k)}.
			%\dkh{ I_n-Z^{(k)}(Z^{(k)})\zz } 
		$
		%\end{equation*}
		Then for any $k \in \N$, it holds that
		\begin{equation*}
			\cL(\{X_i^{(k+1)}\}, Z^{(k)}, \{\Lambda_i^{(k+1)}\}) 
			- \cL(\{X_i^{(k+1)}\}, Z^{(k+1)}, \{\Lambda_i^{(k+1)}\}) \\
			\geq {} c_2 \norm{ G^{(k)} }\fs.
		\end{equation*}
		%where $\bar{Q}^{(k)} \in \Rnn$ is defined in \cref{eq:barQk}.
	\end{corollary}
	
	\begin{proof}
		Recalling the definitions of $Q$ and $\Phi_i (Z)$, we obtain that
		%\begin{equation*}
		%	\sumiid\dkh{ I_n - Z Z\zz } \Phi_i (Z) Z 
		%	= - \sumiid \dkh{I_n - Z Z\zz } A_iA_i\zz Z 
		%	= - \dkh{ I_n - Z Z\zz } AA\zz Z,
		%\end{equation*}
		%resulting in the following equality
		\begin{equation*}
			\Pv{Z} Q Z 
			=  \Pv{Z} \bar{Q} Z  - \sumiid \Pv{Z} \Phi_i (Z) Z 
			= \Pv{Z} \bar{Q} Z 
			+ \Pv{Z} AA\zz Z
			= G.
		\end{equation*}
%		\begin{align*}
%			%& \dkh{ I_n - Z Z\zz } Q^{(k)} Z 
%			 \Pv{Z} Q Z 
%			= & \Pv{Z} \sumiid \dkh{ \beta_i X_i^+ (X_i^+)\zz - \Lambda_i^+ } Z 
%			%& = \dkh{ I_n - Z Z\zz } 
%			%\sumiid \dkh{ \beta_i X_i^+ (X_i^+)\zz - \Phi_i (Z) + \Phi_i (Z) - \Phi_i (X_i^+) } Z  \\
%			=  - \sumiid \Pv{Z} \Phi_i (Z) Z \\
%			& 
%			+\Pv{Z}  \sumiid \dkh{ \beta_i X_i^+ (X_i^+)\zz 
%				+ \Phi_i (Z) - \Phi_i (X_i^+) } Z  
%			%& = \dkh{ I_n - Z Z\zz } \bar{Q}^{(k)} Z 
%			= \Pv{Z} \bar{Q} Z 
%			%+ \dkh{ I_n - Z Z\zz } AA\zz Z = G^{(k)}.
%			+ \Pv{Z} AA\zz Z.
%		\end{align*}
		%which equals $G^{(k)}$ when evaluated at $k$-th iteration.
		This together with \eqref{eq:ps-sub-z-con} completes the proof.
	\end{proof}
	
	%Now based on the above lemmas and corollaries, we can demonstrate 
	%the monotonic non-increasing of the sequence $\{ \cL^{(k)}\}$, where
	%$\cL^{(k)} := \cL ( \{X_i^{(k)}\}, Z^{(k)}, \{\Lambda_i^{(k)}\} )$ 
	%is the augmented Lagrangian function value at iteration $k$.
	Next we show the monotonicity of the sequence of 
	augmented Lagrangian function values $\{ \cL^{(k)}\}$ 
	where $\cL^{(k)} = \cL ( \{X_i^{(k)}\}, Z^{(k)}, \{\Lambda_i^{(k)}\} )$.
	
	\begin{proposition}
		\label{prop:lag-des}
		Suppose $\{ \{X_i^{(k)}\}, Z^{(k)} \}$ is the iterate sequence generated 
		by Algorithm \ref{alg:FAPS} 
		initiated from $( \{X_i^{(0)}\}, Z^{(0)} )$ satisfying \eqref{eq:initial},
		and problem parameters satisfy Assumption \ref{asp:beta-svd}. 
		Then the sequence $\{ \cL^{(k)} \}$ is monotonically non-increasing 
		and, for any $k \in \N$, satisfies  the following two conditions:
		\begin{equation}
			\label{eq:lag-des}
			\cL^{(k)} - \cL^{(k+1)}
			\geq \sumiid 
			\dkh{ \dfrac{1}{4} c_1 \underline{\sigma}^2 \beta_i - \dfrac{1 + \sqrt{2 \rho d}}{\rho d} 
			\norm{A_i}_2^2 } \ditk
			%\norm{ D_i^{(k)} }\fs
			+ c_2 \norm{ G^{(k)} }\fs,
		\end{equation}
		and
		\begin{equation}
			\label{eq:lag-des-z}
			\cL^{(k)} - \cL^{(k+1)}
			\geq c_3 \norm{ \Pv{Z^{(k)}} AA\zz Z^{(k)} }\fs,
		\end{equation}
		where $c_3 > 0$ is a constant.
	\end{proposition}
	
	\begin{proof}
		Combining Corollaries \ref{cor:des-xy}, \ref{cor:des-lambda}, and \ref{cor:des-z}, 
		we can easily verify the inequality \eqref{eq:lag-des}.
		Recalling the condition 
		$\beta_i > 4 (1 + \sqrt{2\rho d}) \norm{A_i}_2^2 / (c_1 \underline{\sigma}^2 \rho d) $ 
		in Assumption \ref{asp:beta-svd}, 
		%we can conclude that $\cL^{(k)} - \cL^{(k + 1)} \geq 0$.
		we can conclude that $\cL - \cL^+ \geq 0$.
		Hence, the sequence $\{ \cL^{(k)} \}$ is monotonically non-increasing.
		It directly follows from the definition of $\bar{Q}$ that
		\begin{equation*}
			\Pv{Z} \bar{Q} Z  %\bar{Q}^{(k)} 
			=  \Pv{Z} \sumiid \dkh{ \beta_i  
				\dsp{X_i^+}{Z}
			%	\dkh{ X_i^+ (X_i^+)\zz - Z Z\zz } 
				+  \Phi_i (Z) - \Phi_i (X_i^+) } Z.
		\end{equation*}
		Together with the triangular inequality and \eqref{eq:dist-xxzz}, we can obtain that
		\begin{equation*}
			%& \norm{ \dkh{ I_n - Z Z\zz } \bar{Q}^{(k)} Z }\ff 
			%&& 
			\norm{ \Pv{Z} \bar{Q} Z }\ff 
			\leq \sumiid ( \beta_i 
			\dist{X_i^+}{Z}
			%\norm{ X_i^+ (X_i^+)\zz - Z Z\zz }\ff 
			+ %\sumiid 
			\norm{ \Phi_i (X_i^+) - \Phi_i (Z) }\ff )%\\
			%& \leq & 
			%\leq 
			%\sumiid \dkh{ \beta_i + 4\norm{A_i}_2^2 } 
			%\dist{X_i^+}{Z}
			\leq \sqrt{\dfrac{1}{2 \rho d}} \sumiid ( \beta_i + 4\norm{A_i}_2^2 )\di. %\norm{ D_i^{(k)} }\ff,
			%\norm{ D_i }\ff,
		\end{equation*}
		And we define a constant 
		$c_4 := \min_{\iid} \hkh{ c_1 \underline{\sigma}^2 \beta_i / 4 - (1 + \sqrt{2 \rho d}) \norm{A_i}_2^2 / ( \rho d ) } > 0$.
		It can be readily verified that
		\begin{equation*}
			%\label{eq:add1}
				\norm{ \Pv{Z} AA\zz Z }\ff 
				= \norm{ G - \Pv{Z} \bar{Q} Z }\ff
				\leq \norm{ G }\ff  + \norm{ \Pv{Z} \bar{Q}Z }\ff %\nonumber\\
				 %\leq 
				 %\norm{ G }\ff + \dfrac{1}{2 \rho d} \sumiid \dkh{ \beta_i + 4\norm{A_i}_2^2 } \di %\norm{ D_i}\ff
				\leq \sqrt{ {(\cL - \cL^+ )}/{c_3} },
		\end{equation*}
		where $c_3 := \dkh{ \sum_{i = 1}^d ( \beta_i + 4\norm{A_i}_2^2 ) / \sqrt{2 \rho d c_4} + \sqrt{1 / c_2} }^{-2} > 0$ is a constant,
%		\begin{equation*}
%			c_3 = \dkh{ 
%				\dfrac{1}{2 \rho d \sqrt{c_4}} \sumiid \dkh{ \beta_i + 4\norm{A_i}_2^2 }
%				+ \dfrac{1}{\sqrt{c_2}}
%			}^{-2}.
%		\end{equation*}
		and the last inequality follows from the facts that, for $\iid$,
		\begin{equation}
			\label{eq:con-xxzz}
			%\norm{ G^{(k)} }\ff
			\norm{ G }\ff
			%\leq \sqrt{ \dfrac{\cL^{(k)} - \cL^{(k + 1)} }{c_2} }
			\leq \sqrt{ {(\cL - \cL^+) }/{c_2} }
			\quad \mbox{and} \quad
			%\norm{ D_i^{(k)} }\ff 
			\di
			%\leq \sqrt{ \dfrac{\cL^{(k)} - \cL^{(k + 1)} }{c_4} },
			\leq \sqrt{ {(\cL - \cL^+)}/{c_4} }.
			%\dkh{ \cL ( \{X_i\}, Z, \{\Lambda_i\} ) - \cL ( \{X_i^+\}, Z^+, \{\Lambda_i^+\} ) }^{\half},
		\end{equation}
		%where $c_4 = \min_{\iid} \hkh{ c_1 \underline{\sigma}^2 \beta_i / 4 - (1 + 2 \rho d) \norm{A_i}_2^2 / (2 \rho^2 d^2) } > 0$
%		\begin{equation*}
%			c_4 = \min\limits_{\iid} \hkh{ 
%				\dfrac{1}{4} c_1 \underline{\sigma}^2 \beta_i 
%				- \dfrac{1 + 2 \rho d}{2 \rho^2 d^2} \norm{A_i}_2^2 
%			} > 0
%		\end{equation*}
		%is a constant. 
		We complete the proof.
	\end{proof}
	
	%Based on the above properties, 
	We are now ready to present the proof of Theorem \ref{thm:global}.
	%which establishes the global convergence result of our proposed algorithm. 
	
	\begin{proof}[Proof of Theorem \ref{thm:global}]
		%In this proof, we denote $\cL^{(k)} = \cL ( \{X_i^{(k)}\}, Z^{(k)}, \{\Lambda_i^{(k)}\} )$ for convenience.
		Since each of $X_i^{(k)}$ or $Z^{(k)}$ is orthonormal for any 
		$\iid$ and $k \in \N$, the whole sequence $\{ \{X_i^{(k)}\}, Z^{(k)} \}$ is naturally bounded.
		Then, it follows from the Bolzano-Weierstrass theorem that 
		this sequence exists an accumulation point $\dkh{ \{X_i^{\ast}\}, Z^{\ast} }$,
		where $X_i^{\ast} \in \stiefel$ and $Z^{\ast} \in \stiefel$.
		In addition, the boundedness of $\{ \Lambda_i^{(k)} \}$ 
		results from the multipliers updating formula  \eqref{eq:ps-mult}. 
		Hence, the lower boundedness of $\{ \cL^{(k)} \}$
		is owing to the continuity of the augmented Lagrangian function.
		Namely, there exists a constant $\underline{L}$ such that
		%\begin{equation*}
		$\cL^{(k)} \geq \underline{L}$
		%\end{equation*}
		holds for all $k \in \N$.
		
		% % %
		%To shorten notation, 
		Let %$$R^{(k)}=\dkh{ I_n - Z^{(k)}(Z^{(k)})\zz } AA\zz Z^{(k)}.$$
		$R^{(k)}=\Pv{Z^{(k)}} AA\zz Z^{(k)}$.
		It follows from \eqref{eq:lag-des-z} and \eqref{eq:con-xxzz} that there hold
		\begin{equation}
			\label{eq:sublinear-opt}
			\begin{aligned}
				\sum\limits_{k=0}^{N-1} \norm{ R^{(k)} }\fs 
				\leq {} & \dfrac{1}{c_3} \sum\limits_{k=0}^{N-1} 
				\dkh{ \cL^{(k)} - \cL^{(k+1)} } %\\
				%= {} & \dfrac{1}{c_3} \dkh{ \cL^{(0)} - \cL^{(G)} } 
				\leq  \dfrac{1}{c_3} \dkh{ \cL^{(0)} - \underline{L} }
			\end{aligned}
		\end{equation}
		and
		\begin{equation}
			\label{eq:sublinear-fea}
			%\begin{aligned}
			\sum\limits_{k = 0}^{N-1} \sumiid \ditk
			%\norm{ D_i^{(k)} }\fs
			\leq  \dfrac{d}{c_4} \sum\limits_{k=0}^{N-1} 
			\dkh{ \cL^{(k)} - \cL^{(k+1)} }
			% = \dfrac{1}{c_4} \dkh{ \cL^{(0)} - \cL^{(G)} } 
			\leq  \dfrac{d}{c_4} \dkh{ \cL^{(0)} - \underline{L} }.
			%\end{aligned}
		\end{equation}
		Taking the limit as $N \to \infty$ on the both sides of
		\eqref{eq:sublinear-opt} and \eqref{eq:sublinear-fea},
		we obtain that
		\begin{equation*}
			\sum\limits_{k = 0}^{\infty} \norm{ R^{(k)} }\fs < \infty 
			\quad  \mbox{and} \quad 
			\sum\limits_{k = 0}^{\infty} \sumiid 
			%\norm{ D_i^{(k)} }\fs 
			\ditk
			< \infty,			
		\end{equation*}
		which further imply
		\begin{equation*}
			\lim\limits_{k\to\infty} \norm{ R^{(k)} }\ff = 0
			\quad  \mbox{and} \quad 
			\lim\limits_{k\to\infty} \sumiid 
			%\norm{ D_i^{(k)} }\ff
			\dik
			 = 0.
		\end{equation*}
		Hence, it holds at any limit point that
		%\begin{equation*}
			%\dkh{ I_n - Z^{\ast} ( Z^{\ast} )\zz} 
			$\Pv{Z^{\ast}}
			AA\zz Z^{\ast} = 0$
			%\quad  \mbox{and} \quad
			and 
			$X_i^{\ast} ( X_i^{\ast} )\zz = Z^{\ast} ( Z^{\ast} )\zz$, 
			for $\iid$.
		%\end{equation*}
		Therefore, $Z^{\ast}$ is a first-order stationary point of the problem \eqref{eq:opt-trace}.
		Finally, it follows from the inequalities \eqref{eq:sublinear-opt} and \eqref{eq:sublinear-fea} that
		\begin{equation*}
			\min\limits_{k = 0, \dotsc, N-1} 
			\hkh{
				\norm{ R^{(k)} }\fs +
				\dfrac{1}{d} \sumiid 
				%\norm{ D_i^{(k)} }\fs
				\ditk
			} 
			\leq \dfrac{1}{N} \sum\limits_{k = 0}^{N - 1}
			\hkh{
				\norm{R^{(k)}}\fs +
				\dfrac{1}{d} \sumiid 
				\ditk
				%\norm{ D_i^{(k)} }\fs
			}
			\leq \frac{C}{N},
		\end{equation*}
		%where $C = (\cL^{(0)} - \underline{L})/c_3 + (\cL^{(0)} - \underline{L})/c_4 > 0$ is a constant.
		where $C = (\cL^{(0)} - \underline{L})(1/c_3 +1/c_4) > 0$ is a constant.
		The proof is completed.
	\end{proof}

\bibliography{library}

\begin{thebibliography}{32}
\providecommand{\natexlab}[1]{#1}
\providecommand{\url}[1]{\texttt{#1}}
\expandafter\ifx\csname urlstyle\endcsname\relax
  \providecommand{\doi}[1]{doi: #1}\else
  \providecommand{\doi}{doi: \begingroup \urlstyle{rm}\Url}\fi

\bibitem[Acar et~al.(2018)Acar, Aksu, Uluagac, and Conti]{Acar2018survey}
Abbas Acar, Hidayet Aksu, A~Selcuk Uluagac, and Mauro Conti.
\newblock {A survey on homomorphic encryption schemes: Theory and
  implementation}.
\newblock \emph{ACM Computing Surveys}, 51\penalty0 (4):\penalty0 1--35, 2018.
\newblock \doi{10.1145/3214303}.

\bibitem[Aharon et~al.(2006)Aharon, Elad, and Bruckstein]{Aharon2006}
Michal Aharon, Michael Elad, and Alfred Bruckstein.
\newblock {K-SVD: An algorithm for designing overcomplete dictionaries for
  sparse representation}.
\newblock \emph{IEEE Transactions on Signal Processing}, 54\penalty0
  (11):\penalty0 4311--4322, 2006.
\newblock \doi{10.1109/TSP.2006.881199}.

\bibitem[Andrews and Patterson(1976)]{Andrews1976}
Harry~C. Andrews and Claude~L. Patterson.
\newblock {Singular value decomposition (SVD) image coding}.
\newblock \emph{IEEE Transactions on Communications}, 24\penalty0 (4):\penalty0
  425--432, 1976.
\newblock \doi{10.1109/TCOM.1976.1093309}.

\bibitem[Cand{\`e}s and Recht(2009)]{Candes2009}
Emmanuel~J. Cand{\`e}s and Benjamin Recht.
\newblock {Exact matrix completion via convex optimization}.
\newblock \emph{Foundations of Computational Mathematics}, 9\penalty0
  (6):\penalty0 717, 2009.
\newblock \doi{10.1007/s10208-009-9045-5}.

\bibitem[Chai et~al.(2021)Chai, Wang, Fu, Zhang, Chen, and
  Yang]{Chai2021federated}
Di~Chai, Leye Wang, Lianzhi Fu, Junxue Zhang, Kai Chen, and Qiang Yang.
\newblock {Practical lossless federated singular vector decomposition over
  billion-scale data}.
\newblock \emph{arXiv:2105.08925}, 2021.
\newblock URL \url{https://arxiv.org/abs/2105.08925}.

\bibitem[Deerwester et~al.(1990)Deerwester, Dumais, Furnas, Landauer, and
  Harshman]{Deerwester1990}
Scott Deerwester, Susan~T. Dumais, George~W. Furnas, Thomas~K. Landauer, and
  Richard Harshman.
\newblock Indexing by latent semantic analysis.
\newblock \emph{Journal of the American Society for Information Science},
  41\penalty0 (6):\penalty0 391--407, 1990.
\newblock
  \doi{10.1002/(SICI)1097-4571(199009)41:6%3C391::AID-ASI1%3E3.0.CO;2-9}.

\bibitem[Dwork et~al.(2014)Dwork, Roth, et~al.]{Dwork2014algorithmic}
Cynthia Dwork, Aaron Roth, et~al.
\newblock The algorithmic foundations of differential privacy.
\newblock \emph{Foundations and Trends{\textregistered} in Theoretical Computer
  Science}, 9\penalty0 (3--4):\penalty0 211--407, 2014.
\newblock \doi{10.1561/0400000042}.

\bibitem[Eckstein and Silva(2013)]{Eckstein2013}
Jonathan Eckstein and Paulo J.~S. Silva.
\newblock {A practical relative error criterion for augmented Lagrangians}.
\newblock \emph{Mathematical Programming}, 141\penalty0 (1-2):\penalty0
  319--348, 2013.
\newblock \doi{10.1007/s10107-012-0528-9}.

\bibitem[Gang and Bajwa(2021)]{Gang2021fast}
Arpita Gang and Waheed~U. Bajwa.
\newblock {FAST-PCA: A fast and exact algorithm for distributed principal
  component analysis}.
\newblock \emph{arXiv:2108.12373}, 2021.
\newblock URL \url{https://arxiv.org/abs/2108.12373}.

\bibitem[Gang and Bajwa(2022)]{Gang2022linearly}
Arpita Gang and Waheed~U. Bajwa.
\newblock {A linearly convergent algorithm for distributed principal component
  analysis}.
\newblock \emph{Signal Processing}, 193:\penalty0 108408, 2022.
\newblock ISSN 0165-1684.
\newblock \doi{10.1016/j.sigpro.2021.108408}.

\bibitem[Gang et~al.(2019)Gang, Raja, and Bajwa]{Gang2019}
Arpita Gang, Haroon Raja, and Waheed~U. Bajwa.
\newblock {Fast and communication-efficient distributed PCA}.
\newblock In \emph{International Conference on Acoustics, Speech and Signal
  Processing (ICASSP)}, pages 7450--7454. IEEE, 2019.
\newblock \doi{10.1109/ICASSP.2019.8683095}.

\bibitem[Gao et~al.(2018)Gao, Liu, Chen, and Yuan]{Gao2018}
Bin Gao, Xin Liu, Xiaojun Chen, and Ya-Xiang Yuan.
\newblock {A new first-order algorithmic framework for optimization problems
  with orthogonality constraints}.
\newblock \emph{SIAM Journal on Optimization}, 28\penalty0 (1):\penalty0
  302--332, 2018.
\newblock ISSN 1052-6234.
\newblock \doi{10.1137/16M1098759}.

\bibitem[Grammenos et~al.(2020)Grammenos, Mendoza~Smith, Crowcroft, and
  Mascolo]{Grammenos2020federated}
Andreas Grammenos, Rodrigo Mendoza~Smith, Jon Crowcroft, and Cecilia Mascolo.
\newblock {Federated principal component analysis}.
\newblock \emph{Advances in Neural Information Processing Systems},
  33:\penalty0 6453--6464, 2020.
\newblock URL
  \url{https://proceedings.neurips.cc/paper/2020/hash/47a658229eb2368a99f1d032c8848542-Abstract.html}.

\bibitem[Li et~al.(2021)Li, Wang, Chen, and Zhang]{Li2021communication}
Xiang Li, Shusen Wang, Kun Chen, and Zhihua Zhang.
\newblock {Communication-efficient distributed SVD via local power iterations}.
\newblock In \emph{International Conference on Machine Learning (ICML)}, pages
  6504--6514. PMLR, 2021.
\newblock URL \url{https://proceedings.mlr.press/v139/li21u.html}.

\bibitem[Liu and Tang(2019)]{Liu2019privacy}
Bowen Liu and Qiang Tang.
\newblock {Privacy-preserving decentralised singular value decomposition}.
\newblock In \emph{International Conference on Information and Communications
  Security (ICICS)}, pages 703--721. Springer, 2019.
\newblock \doi{10.1007/978-3-030-41579-2_41}.

\bibitem[Liu et~al.(2013)Liu, Wen, and Zhang]{Liu2013}
Xin Liu, Zaiwen Wen, and Yin Zhang.
\newblock {Limited memory block Krylov subspace optimization for computing
  dominant singular value decompositions}.
\newblock \emph{SIAM Journal on Scientific Computing}, 35\penalty0
  (3):\penalty0 A1641--A1668, 2013.
\newblock \doi{10.1137/120871328}.

\bibitem[Liu et~al.(2015)Liu, Wen, and Zhang]{Liu2015b}
Xin Liu, Zaiwen Wen, and Yin Zhang.
\newblock {An efficient Gauss--Newton algorithm for symmetric low-rank product
  matrix approximations}.
\newblock \emph{SIAM Journal on Optimization}, 25\penalty0 (3):\penalty0
  1571--1608, 2015.
\newblock \doi{10.1137/140971464}.

\bibitem[Liu et~al.(2019)Liu, Liu, and Ma]{Liu2019}
Ya-Feng Liu, Xin Liu, and Shiqian Ma.
\newblock {On the nonergodic convergence rate of an inexact augmented
  Lagrangian framework for composite convex programming}.
\newblock \emph{Mathematics of Operations Research}, 44\penalty0 (2):\penalty0
  632--650, 2019.
\newblock \doi{10.1287/moor.2018.0939}.

\bibitem[Lou et~al.(2017)Lou, Yu, Wang, and Yi]{Lou2017}
Youcheng Lou, Lean Yu, Shouyang Wang, and Peng Yi.
\newblock {Privacy preservation in distributed subgradient optimization
  algorithms}.
\newblock \emph{IEEE Transactions on Cybernetics}, 48\penalty0 (7):\penalty0
  2154--2165, 2017.
\newblock \doi{10.1109/TCYB.2017.2728644}.

\bibitem[McMahan et~al.(2017)McMahan, Moore, Ramage, Hampson, and
  Arcas]{Mcmahan2017communication}
Brendan McMahan, Eider Moore, Daniel Ramage, Seth Hampson, and Blaise Aguera~y
  Arcas.
\newblock {Communication-efficient learning of deep networks from decentralized
  data}.
\newblock In \emph{International Conference on Artificial Intelligence and
  Statistics (AISTATS)}, volume~54, pages 1273--1282. PMLR, 2017.
\newblock URL \url{https://proceedings.mlr.press/v54/mcmahan17a.html}.

\bibitem[Moore(1981)]{Moore1981}
Bruce Moore.
\newblock Principal component analysis in linear systems: controllability,
  observability, and model reduction.
\newblock \emph{IEEE Transactions on Automatic Control}, 26\penalty0
  (1):\penalty0 17--32, 1981.
\newblock \doi{10.1109/TAC.1981.1102568}.

\bibitem[Rutishauser(1970)]{Rutishauser1970}
Heinz Rutishauser.
\newblock {Simultaneous iteration method for symmetric matrices}.
\newblock \emph{Numerische Mathematik}, 16\penalty0 (3):\penalty0 205--223,
  1970.
\newblock \doi{10.1007/BF02219773}.

\bibitem[Schizas and Aduroja(2015)]{Schizas2015}
Ioannis~D. Schizas and Abiodun Aduroja.
\newblock {A distributed framework for dimensionality reduction and denoising}.
\newblock \emph{IEEE Transactions on Signal Processing}, 63\penalty0
  (23):\penalty0 6379--6394, 2015.
\newblock \doi{10.1109/TSP.2015.2465300}.

\bibitem[Stewart(1976)]{Stewart1976}
Gilbert~W. Stewart.
\newblock {Simultaneous iteration for computing invariant subspaces of
  non-Hermitian matrices}.
\newblock \emph{Numerische Mathematik}, 25\penalty0 (2):\penalty0 123--136,
  1976.
\newblock \doi{10.1007/BF01462265}.

\bibitem[Stewart and Jennings(1981)]{Stewart1981}
William~J. Stewart and Alan Jennings.
\newblock {A simultaneous iteration algorithm for real matrices}.
\newblock \emph{ACM Transactions on Mathematical Software}, 7\penalty0
  (2):\penalty0 184--198, 1981.
\newblock \doi{10.1145/355945.355948}.

\bibitem[Turk and Pentland(1991)]{Turk1991}
Matthew~A. Turk and Alex~P. Pentland.
\newblock Face recognition using eigenfaces.
\newblock In \emph{Computer Society Conference on Computer Vision and Pattern
  Recognition (CVPR)}, pages 586--591. IEEE, 1991.
\newblock \doi{10.1109/CVPR.1991.139758}.

\bibitem[Wang and Liu(2022)]{Wang2022decentralized}
Lei Wang and Xin Liu.
\newblock {Decentralized optimization over the Stiefel manifold by an
  approximate augmented Lagrangian function}.
\newblock \emph{IEEE Transactions on Signal Processing}, 70:\penalty0
  3029--3041, 2022.
\newblock \doi{10.1109/TSP.2022.3182883}.

\bibitem[Wang et~al.(2021)Wang, Gao, and Liu]{Wang2021multipliers}
Lei Wang, Bin Gao, and Xin Liu.
\newblock {Multipliers correction methods for optimization problems over the
  Stiefel manifold}.
\newblock \emph{CSIAM Transactions on Applied Mathematics}, 2\penalty0
  (3):\penalty0 508--531, 2021.
\newblock ISSN 2708-0579.
\newblock \doi{10.4208/csiam-am.SO-2020-0008}.

\bibitem[Wang et~al.(2019)Wang, Yin, and Zeng]{Wang2019global}
Yu~Wang, Wotao Yin, and Jinshan Zeng.
\newblock {Global convergence of ADMM in nonconvex nonsmooth optimization}.
\newblock \emph{Journal of Scientific Computing}, 78\penalty0 (1):\penalty0
  29--63, 2019.
\newblock \doi{10.1007/s10915-018-0757-z}.

\bibitem[Ye and Zhang(2021)]{Ye2021}
Haishan Ye and Tong Zhang.
\newblock {DeEPCA: Decentralized exact PCA with linear convergence rate}.
\newblock \emph{Journal of Machine Learning Research}, 22\penalty0
  (238):\penalty0 1--27, 2021.
\newblock URL \url{http://jmlr.org/papers/v22/21-0298.html}.

\bibitem[Zhang et~al.(2018)Zhang, Ahmad, and Wang]{Zhang2018a}
Chunlei Zhang, Muaz Ahmad, and Yongqiang Wang.
\newblock {ADMM based privacy-preserving decentralized optimization}.
\newblock \emph{IEEE Transactions on Information Forensics and Security},
  14\penalty0 (3):\penalty0 565--580, 2018.
\newblock \doi{10.1109/TIFS.2018.2855169}.

\bibitem[Zhang et~al.(2020)Zhang, Ma, and Zhang]{Zhang2020primal}
Junyu Zhang, Shiqian Ma, and Shuzhong Zhang.
\newblock {Primal-dual optimization algorithms over Riemannian manifolds: an
  iteration complexity analysis}.
\newblock \emph{Mathematical Programming}, 184\penalty0 (1):\penalty0 445--490,
  2020.
\newblock \doi{10.1007/s10107-019-01418-8}.

\end{thebibliography}

\end{document}